\documentclass{amsart}
\usepackage{amsfonts}
\usepackage{amssymb}
\usepackage{amsmath}
\usepackage{amsthm}
\usepackage{cite}

\newtheorem{theorem}{Theorem}[section]
\newtheorem{proposition}[theorem]{Proposition}
\newtheorem{lemma}[theorem]{Lemma}

\theoremstyle{definition}

\newtheorem{problem}[theorem]{Problem}
\newtheorem{remark}[theorem]{Remark}

\newcommand{\<}{\left\langle}
\renewcommand{\>}{\right\rangle}
\newcommand{\eps}{\epsilon}

\newcommand{\rr}{ \mathbb{R}}
\newcommand{\di}{ \mathbf{d}}
\newcommand{\ti}{ \mathbf{t}}
\newcommand{\si}{ \mathbf{s}}
\newcommand{\pc}{p}

\newcommand{\dR}{\mathbb{R}}

\newcommand{\rmd}{\mathrm{d}}

\newcommand{\olB}{\overline{B}}

\DeclareMathOperator{\dist}{dist}

\DeclareMathOperator{\diam}{diam}

\DeclareMathOperator{\opdiv}{div}

\newcommand{\dint}{\,\rmd}
\newcommand{\ssm}{\backslash}

\newcommand{\lr}[3]{#1#3#2}
\newcommand{\xlr}[3]{\left#1#3\right#2}
\newcommand{\biglr}[3]{\bigl#1#3\bigr#2}

\newcommand{\bigglr}[3]{\biggl#1#3\biggr#2}

\newcommand{\abs}[1]{\lr\lvert\rvert{#1}}
\newcommand{\xabs}[1]{\xlr\lvert\rvert{#1}}
\newcommand{\bigabs}[1]{\biglr\lvert\rvert{#1}}

\newcommand{\biggabs}[1]{\bigglr\lvert\rvert{#1}}

\newcommand{\norm}[1]{\lr\lVert\rVert{#1}}

\newcommand{\fracwithdelims}[4]{\genfrac{#1}{#2}{}{}{#3}{#4}}
\newcommand{\coloneqq}{:=}

 \allowdisplaybreaks[1]

\begin{document}
\title[Boundary clustered layers near the higher critical exponents]{Boundary clustered layers near the higher critical exponents}
\author{Nils Ackermann}
\address{Instituto de Matem\'{a}ticas, Universidad Nacional Aut\'{o}noma de M\'{e}xico,
Circuito Exterior, C.U., 04510 M\'{e}xico D.F., Mexico.\smallskip}
\email{nils@ackermath.info}
\author{M\'{o}nica Clapp}
\address{Instituto de Matem\'{a}ticas, Universidad Nacional Aut\'{o}noma de M\'{e}xico,
Circuito Exterior, C.U., 04510 M\'{e}xico D.F., Mexico.\smallskip}
\email{monica.clapp@im.unam.mx}
\author{Angela Pistoia}
\address{Dipartimento di Metodi e Modelli Matematici, Universit\`{a} "La Sapienza" di
Roma, via Antonio Scarpa 16, 00161 Roma, Italia.}
\email{pistoia@dmmm.uniroma1.it}
\thanks{This research was partially supported by CONACYT grant 129847 and
PAPIIT-DGAPA-UNAM grant IN106612 (Mexico), and by exchange funds of the
Universit\`{a} \textquotedblleft La Sapienza\textquotedblright\ di Roma (Italy).}
\date{October 2012}

\begin{abstract}
We consider the supercritical problem
\begin{equation*}
-\Delta u=\left\vert u\right\vert ^{p-2}u\text{ \ in }\Omega,\quad u=0\text{
\ on }\partial\Omega,
\end{equation*}
where $\Omega$ is a bounded smooth domain in $\mathbb{R}^{N}$ and $p$ smaller
than the critical exponent $2_{N,k}^{\ast}:=\frac{2(N-k)}{N-k-2}$ for the
Sobolev embedding of $H^{1}(\mathbb{R}^{N-k})$ in $L^{q}(\mathbb{R}^{N-k})$,
$1\leq k\leq N-3.$ We show that in some suitable domains $\Omega$ there are
positive and sign changing solutions with positive and negative layers which
concentrate along one or several $k$-dimensional submanifolds of
$\partial\Omega$ as $p$ approaches $2_{N,k}^{\ast}$ from below.

\textsc{Key words: }Nonlinear elliptic boundary value problem; critical and
supercritical exponents; existence of positive and sign changing solutions.

\textsc{MSC2010: }35J60, 35J20.

\end{abstract}
\maketitle

\section{Introduction}

Consider the classical Lane-Emden-Fowler problem
\begin{equation}
\Delta v+|v|^{p -2}v=0\quad\text{in }\ \mathcal{D},\qquad v=0\quad\text{on
}\partial\mathcal{D}, \label{sup-eq}
\end{equation}
where $\mathcal{D}$ is a bounded smooth domain in $\mathbb{R}^{N}$ and $p >2.
$

It is well known that when $p$ is smaller than the critical Sobolev exponent
$2^{\ast}:=\frac{2N}{N-2},$ compactness of the Sobolev embedding ensures the
existence of at least one positive solution and infinitely many sign changing
solutions. In contrast, existence of solutions to problem (\ref{sup-eq}) when
$p\geq2^{\ast}$ is a delicate issue. Pohozhaev's identity \cite{MR0192184}
implies that problem (\ref{sup-eq}) does not have a nontrivial solution if the
domain $\mathcal{D}$ is strictly starshaped. On the other hand, Kazdan and
Warner showed in \cite{MR0477445} that if the domain $\mathcal{D}$ is an
annulus, problem (\ref{sup-eq}) has infinitely many radial solutions.

For the critical case $p=2^{\ast}$ Bahri and Coron \cite{MR89c:35053} proved
that a positive solution of (\ref{sup-eq}) exists if the domain $\mathcal{D}$
has nontrivial reduced homology with $\mathbb{Z}/2$-coefficients. Moreover, it
was proved by Ge, Musso and Pistoia \cite{MR2754050} and Musso and Pistoia
\cite{MR2281450} that, if $\mathcal{D}$ has a small hole, problem
(\ref{sup-eq}) has many sign changing solutions, whose number increases as the
diameter of the hole decreases. Multiplicity results are also available for
domains which are not small perturbations of a given domain, but have enough,
possibly finite, symmetries, as proved by Clapp and Pacella \cite{MR2395127}
and Clapp and Faya \cite{clapp/faya:2012}.

The almost critical case $p=2^{\ast}\pm\epsilon,$ with $\epsilon$ positive and
small enough, has been widely studied. The slightly subcritical case
$p=2^{\ast}-\epsilon$ was considered by Bahri, Li and Rey \cite{MR1384837} and
Rey \cite{MR1040954}, who showed the existence of positive solutions which
blow-up at one or more points of $\mathcal{D}$ as $\epsilon\rightarrow0$. A
large number of sign changing solutions with simple or multiple positive and
negative blow-up points were constructed by Bartsch, Micheletti and Pistoia
\cite{MR2232205}, Musso and Pistoia \cite{MR2579374}, and Pistoia and Weth
\cite{MR2310698}. For the slightly supercritical case $p=2^{\ast}+\epsilon$
existence and nonexistence of positive solutions with one or more blow-up
points has been established by Ben Ayed, El Mehdi, Grossi and Rey
\cite{MR1979006}, Pistoia and Rey \cite{MR2238023}, and del Pino, Felmer and
Musso \cite{MR1966257}.

Unlike the critical case, in the supercritical case $p>2^{\ast}$ the existence
of a nontrivial homology class in $\mathcal{D}$ does not guarantee the
existence of a nontrivial solution to (\ref{sup-eq}). In fact, for each
integer $k$ such that $1\leq k\leq N-3,$ Passaseo \cite{MR1220984, MR1306576}
exhibited a bounded smooth domain in $\mathbb{R}^{N},$ homotopically
equivalent to the $k$-dimensional sphere, in which problem (\ref{sup-eq}) does
not have a nontrivial solution for $p\geq{2}_{N,k}^{\ast}:={\frac
{2(N-k)}{N-k-2}}.$ Note that ${2}_{N,k}^{\ast}$ is the critical Sobolev
exponent in dimension $N-k.$ Examples of domains with richer homology were
recently given by Clapp, Faya and Pistoia \cite{clapp/faya/pistoia:2012},
where it was shown that for $p>{2}_{N,k}^{\ast}$ there are bounded smooth
domains in $\mathbb{R}^{N}$ whose cup-length is $k+1,$ in which problem
(\ref{sup-eq}) does not have a nontrivial solution. On the other hand, for
$p={2}_{N,k}^{\ast}$ existence of infinitely many solutions in some domains
has been recently established by Wei and Yan \cite{MR2832637}. Further
multiplicity results may be found in \cite{clapp/faya/pistoia:2012}.

In \cite{MR2734352} del Pino, Musso and Pacard considered the case
$p={2}_{N,1}^{\ast}-\epsilon$ and proved that for some suitable domains
$\mathcal{D}$, if $\epsilon$ is positive, small enough and different from an
explicit set of values, problem (\ref{sup-eq}) has a positive solution which
concentrates along a $1$-dimensional submanifold of the boundary
$\partial\mathcal{D}.$ In the same paper, the authors ask the question whether
one can find solutions which concentrate at a $k$-dimensional submanifold for
$p$ slightly below ${2}_{N,k}^{\ast}.$ More precisely, they ask the following:

\begin{problem}
Given $1\leq k\leq N-3,$ are there domains $\mathcal{D}$ in which problem
\emph{(\ref{sup-eq})} has a positive solution $v_{p}$ for each $p<{2}
_{N,k}^{\ast}$ with the property that these solutions concentrate along a
$k$-dimensional submanifold of the boundary $\partial\mathcal{D}$ as
$p\rightarrow{2}_{N,k}^{\ast}?$
\end{problem}

Having in mind that when $p $ approaches the first critical exponent $2^{\ast
}$ from below a large number of sign changing solutions exist, another
question arises naturally:

\begin{problem}
Given $1\leq k\leq N-3,$ are there domains $\mathcal{D}$ in which problem
\emph{(\ref{sup-eq})} has a sign changing solution $v_{p}$ for each
$p<{2}_{N,k}^{\ast}$ with the property that these solutions concentrate along
a $k$-dimensional submanifold of the boundary $\partial\mathcal{D}$ as
$p\rightarrow{2}_{N,k}^{\ast}?$
\end{problem}

In this paper, we give a positive answer to both questions. In particular, for
each set of positive integers $k_{1},\ldots,k_{m}$ with $k:=k_{1}+\cdots
+k_{m}\leq N-3$ we exhibit domains $\mathcal{D}$ in which problem
(\ref{sup-eq}) has a positive solution for each $p<{2}_{N,k}^{\ast}$ and, as
$p\rightarrow{2}_{N,k}^{\ast},$ these solutions concentrate along a
$k$-dimensional submanifold $M$ of the boundary $\partial\mathcal{D}$ which is
diffeomorphic to the product of spheres $\mathbb{S}^{k_{1}}\times\cdots
\times\mathbb{S}^{k_{m}}.$ Moreover, problem (\ref{sup-eq}) has also a sign
changing solution with a positive and a negative layer, both of which
concentrate along $M$ as $p\rightarrow{2}_{N,k}^{\ast}.$ This follows from our
main results, which we next state.

Fix $k_{1},\ldots,k_{m}\in\mathbb{N}$ with $k:=k_{1}+\cdots+k_{m}\leq N-3$ and
a bounded smooth domain $\Omega$ in $\mathbb{R}^{N-k}$ such that
\begin{equation}
\overline{\Omega}\subset\{\left(  x_{1},\ldots,x_{m},x^{\prime}\right)
\in\mathbb{R}^{m}\times\mathbb{R}^{N-k-m}:x_{i}>0,\text{ }i=1,\ldots,m\}.
\label{omega}
\end{equation}
Set
\begin{equation}
\mathcal{D}:=\{(y^{1},\ldots,y^{m},z)\in\mathbb{R}^{k_{1}+1}\times\cdots
\times\mathbb{R}^{k_{m}+1}\times\mathbb{R}^{N-k-m}:\left(  \left\vert
y^{1}\right\vert ,\ldots,\left\vert y^{m}\right\vert ,z\right)  \in\Omega\}.
\label{D}
\end{equation}
$\mathcal{D}$ is a bounded smooth domain in $\mathbb{R}^{N}$ which is
invariant under the action of the group $\Gamma:=O(k_{1}+1)\times\cdots\times
O(k_{m}+1)$ on $\mathbb{R}^{N}$ given by
\begin{equation*}
(g_{1},\ldots,g_{m})(y^{1},\ldots,y^{m},z):=(g_{1}y^{1},\ldots,g_{m}y^{m},z).
\end{equation*}
for every $g_{i}\in O(k_{i}+1),$ $y^{i}\in\mathbb{R}^{k_{i}+1},$
$z\in\mathbb{R}^{N-k-m}.$ Here, as usual, $O(d)$ denotes the group of all
linear isometries of $\mathbb{R}^{d}.$ For $p={2}_{N,k}^{\ast}-\epsilon$ we
shall look for $\Gamma$-invariant solutions to problem (\ref{sup-eq}), i.e.
solutions $v$ of the form
\begin{equation}
v(y^{1},\ldots,y^{m},z)=u(\left\vert y^{1}\right\vert ,\ldots,\left\vert
y^{m}\right\vert ,z). \label{inv}
\end{equation}
A simple calculation shows that $v$ solves problem (\ref{sup-eq}) if and only
if $u$ solves
\begin{equation*}
-\Delta u-\sum_{i=1}^{m}\frac{k_{i}}{x_{i}}\frac{\partial u}{\partial x_{i}
}=|u|^{p-2}u\quad\text{in}\ \Omega,\qquad u=0\quad\text{on}\ \partial\Omega.
\end{equation*}
This problem can be rewritten as
\begin{equation*}
-\text{div}(a(x)\nabla u)=a(x)|u|^{p-2}u\quad\text{in}\ \Omega,\qquad
u=0\quad\text{on}\ \partial\Omega,
\end{equation*}
where $a(x_{1},\ldots,x_{N-k}):=x_{1}^{k_{1}}\cdots x_{m}^{k_{m}}.$ Note that
${2}_{N,k}^{\ast}$ is the critical exponent in dimension $n:=N-k$ which is the
dimension of $\Omega.$

Thus, we are lead to study the more general almost critical problem
\begin{equation}
-\text{div}(a(x)\nabla u)=a(x)\left\vert u\right\vert ^{{\frac{4}{n-2}
}-\epsilon}u\quad\text{in}\ \Omega,\qquad u=0\quad\text{on}\ \partial\Omega,
\label{p1}
\end{equation}
where $\Omega$ is a bounded smooth domain in $\mathbb{R}^{n},$ $n\geq3,$
$\epsilon$ is a positive parameter, and $a\in\mathcal{C}^{2}(\overline{\Omega
})$ is strictly positive on $\overline{\Omega}.$

This is a subcritical problem, so standard variational methods yield one
positive and infinitely many sign changing solutions to problem (\ref{p1}) for
every $\epsilon\in(0,{\frac{4}{n-2})}$, cf. Proposition 4.1 in
\cite{clapp/faya/pistoia:2012}. Our goal is to construct solutions
$u_{\epsilon}$ with positive and negative bubbles which accumulate at some
points $\xi_{1},\ldots,\xi_{\kappa}$ of $\partial\Omega$ as $\epsilon
\rightarrow0.$ They correspond, via (\ref{inv}), to $\Gamma$-invariant
solutions $v_{\epsilon}$ of problem (\ref{sup-eq}) with positive and negative
layers which accumulate along the $k$-dimensional submanifolds
\begin{equation*}
M_{j}:=\{(y^{1},\ldots,y^{m},z)\in\mathbb{R}^{k_{1}+1}\times\cdots
\times\mathbb{R}^{k_{m}+1}\times\mathbb{R}^{N-k-m}:\left(  \left\vert
y^{1}\right\vert ,\ldots,\left\vert y^{m}\right\vert ,z\right)  =\xi_{j}\}
\end{equation*}
of the boundary of $\mathcal{D}$ as $\epsilon\rightarrow0.$ Note that each
$M_{j}$ is diffeomorphic to $\mathbb{S}^{k_{1}}\times\cdots\times
\mathbb{S}^{k_{m}}$ where $\mathbb{S}^{d}$ is the unit sphere in
$\mathbb{R}^{d+1}.$

We will assume one of the following conditions.

\begin{enumerate}
\item[$(a1)$] There exist $\kappa$ nondegenerate critical points $\xi
_{1},\dots,\xi_{\kappa}\in\partial\Omega$ of the restriction of $a$ to
$\partial\Omega$ such that
\begin{equation*}
\langle\nabla a(\xi_{i}),\nu(\xi_{i})\rangle>0\qquad\forall i=1,\dots,\kappa,
\end{equation*}
where $\nu(\xi_{i})$ is the inward pointing unit normal to $\partial\Omega$ at
$\xi_{i}.$

\item[$(a2)$] There exists a critical point $\xi_{0}\in\partial\Omega$ of the
restriction of $a$ to $\partial\Omega$ such that $\langle\nabla a(\xi_{0}
),\nu(\xi_{0})\rangle>0,$ and vectors $\tau_{1},\dots,\tau_{n-1}\in
\mathbb{R}^{n}$ such that the set $\{\nu(\xi_{0}),\tau_{1},\dots,\tau_{n-1}\}$
is orthonormal and $\Omega$ and $a$ are invariant with respect to the
reflection $\varrho_{i}$ on each of the hyperplanes $\xi_{0}+\{\tau_{i}=0\}$,
i.e.
\begin{equation*}
\varrho_{i}(x)\in\Omega\text{\qquad and\qquad}a(\varrho_{i}
(x))=a(x)\text{\qquad}\forall x\in\Omega,
\end{equation*}
$i=1,...,n-1,$ where
\begin{align*}
&  \varrho_{i}(\xi_{0}+\left\langle x,\nu\right\rangle \nu+\left\langle
x,\tau_{1}\right\rangle \tau_{1}+\cdots+\left\langle x,\tau_{i}\right\rangle
\tau_{i}+\cdots+\left\langle x,\tau_{n-1}\right\rangle \tau_{n-1})\\
&  =\xi_{0}+\left\langle x,\nu\right\rangle \nu+\left\langle x,\tau
_{1}\right\rangle \tau_{1}+\cdots-\left\langle x,\tau_{i}\right\rangle
\tau_{i}+\cdots+\left\langle x,\tau_{n-1}\right\rangle \tau_{n-1}.
\end{align*}
and $\nu:=\nu(\xi_{0})$ is the inward pointing unit normal to $\partial\Omega$
at $\xi_{0}.$
\end{enumerate}

For each $\delta>0,\ \xi\in\mathbb{R}^{n},$\ we consider the standard bubble
\begin{equation*}
U_{\delta,\xi}(x):=[n(n-2)]^{\frac{n-2}{4}}{\frac{\delta^{\frac{n-2}{2}}
}{\left(  \delta^{2}+|x-\xi|^{2}\right)  ^{\frac{n-2}{2}}}.}
\end{equation*}
We prove the following results.

\begin{theorem}
\label{main3}Assume that $(a1)$ holds true. Then there exists $\epsilon_{0}>0
$ such that, for each $\lambda_{1},\ldots,\lambda_{\kappa}\in\{0,1\}$ and
$\epsilon\in(0,\epsilon_{0}),$ problem \emph{(\ref{p1})} has a solution
$u_{\epsilon}$ which satisfies
\begin{equation*}
u_{\epsilon}(x)=\sum\limits_{i=1}^{\kappa}(-1)^{\lambda_{i}}U_{\delta
_{i,\epsilon}, \xi_{i,\epsilon}}(x)+o(1)\qquad\text{in }D^{1,2}(\Omega),
\end{equation*}
with
\begin{equation*}
\epsilon^{-{\frac{n-1}{n-2}}}\delta_{i,\epsilon}\rightarrow d_{i}
>0\qquad\text{and}\qquad\xi_{i,\epsilon}\rightarrow\xi_{i}\in\partial\Omega,
\end{equation*}
for each $i=1,\dots,\kappa,$ as $\epsilon\rightarrow0.$
\end{theorem}

\begin{theorem}
\label{main4}Assume that $(a2)$ holds true. Then there exists $\epsilon_{0}>0
$ such that, for each $\epsilon\in(0,\epsilon_{0}),$ problem \emph{(\ref{p1})}
has a sign changing solution $u_{\epsilon}$ which is invariant with respect to
each reflection $\varrho_{i},$\ $i=1,...,n-1,$ and satisfies
\begin{equation*}
u_{\epsilon}(x)=U_{\delta_{1,\epsilon},\xi_{1,\epsilon}}(x)-U_{\delta
_{2,\epsilon},\xi_{2,\epsilon}}(x)+o(1)\qquad\text{in }D^{1,2}(\Omega),
\end{equation*}
with
\begin{equation*}
\epsilon^{-{\frac{n-1}{n-2}}}\delta_{i,\epsilon}\rightarrow d_{i}>0,\quad
\xi_{i,\epsilon}=\xi_{0}+\epsilon t_{i,\epsilon}\nu(\xi_{0})\quad
\hbox{and}\quad t_{i,\epsilon}\rightarrow t_{i}>0,
\end{equation*}
for each $i=1,2,$ as $\epsilon\rightarrow0.$
\end{theorem}

Theorem \ref{main4} states the existence of a sign changing solution whose two
blow-up points (one positive and one negative) collapse to the same point
$\xi_{0}$ of the boundary of $\Omega$ under the symmetry assumption $(a2)$.

Some interesting questions arise:

\begin{problem}
Is it possible to find sign changing solutions with $k\geq3$ blow-up points
with alternating sign which collapse to the point $\xi_{0}$?
\end{problem}

\begin{problem}
Is it possible to find a sign changing solution with one positive and one
negative blow-up point which collapse to the point $\xi_{0}$ in the more
general case when $\xi_{0}$ is a nondegenerate critical point of $a$
constrained to $\partial\Omega$ such that $\langle\nabla a(\xi_{0}),\nu
(\xi_{0})\rangle>0,$ without any symmetry assumption?
\end{problem}

The reason for including the symmetry assumption $(a2)$ in Theorem \ref{main4}
is that it allows to simplify the computations considerably (see Remark
\ref{rem}).

\medskip

In the following two theorems we assume we are given $k_{1},\ldots,k_{m}
\in\mathbb{N}$ with $k:=k_{1}+\cdots+k_{m}\leq N-3$ and a bounded smooth
domain $\Omega$ in $\mathbb{R}^{N-k}$ which satisfies (\ref{omega}). We set
$a(x_{1},\ldots,x_{N-k}):=x_{1}^{k_{1}}\cdots x_{m}^{k_{m}},$ $\ \mathcal{D}$
as in (\ref{D}), $p =2_{N,k}^{\ast}-\epsilon,$ $\Gamma:=O(k_{1}+1)\times
\cdots\times O(k_{m}+1)$ and
\begin{equation*}
\widetilde{U}_{\delta,\xi}(y^{1},\ldots,y^{m},z):=U_{\delta,\xi}(\left\vert
y^{1}\right\vert ,\ldots,\left\vert y^{m}\right\vert ,z)
\end{equation*}
for $\delta>0,\ \xi\in\mathbb{R}^{N-k}.$

\begin{theorem}
\label{main1}Assume that $(a1)$ holds true for $a$ and $\Omega$ as above. Then
there exists $\epsilon_{0}>0$ such that, for each $\lambda_{1},\ldots
,\lambda_{\kappa}\in\{0,1\}$ and $\epsilon\in(0,\epsilon_{0}),$ problem
\emph{(\ref{sup-eq})} has a $\Gamma$-invariant solution $v_{\epsilon}$ which
satisfies
\begin{equation*}
v_{\epsilon}(x)=\sum\limits_{i=1}^{\kappa}(-1)^{\lambda_{i}}\widetilde
{U}_{\delta_{i,\epsilon},\xi_{i,\epsilon}}(x)+o(1)\qquad\text{in }
D^{1,2}(\mathcal{D}),
\end{equation*}
with
\begin{equation*}
\epsilon^{-{\frac{n-1}{n-2}}}\delta_{i,\epsilon}\rightarrow d_{i}
>0\qquad\text{and}\qquad\xi_{i,\epsilon}\rightarrow\xi_{i}\in\partial\Omega,
\end{equation*}
for each $i=1,\dots,\kappa,$ as $\epsilon\rightarrow0.$
\end{theorem}

\begin{theorem}
\label{main2}Assume that $(a2)$ holds true for $a$ and $\Omega$ as above. Then
there exists $\epsilon_{0}>0$ such that, for each $\epsilon\in(0,\epsilon
_{0}),$ problem \emph{(\ref{sup-eq})} has a $\Gamma$-invariant sign changing
solution $v_{\epsilon}$ which satisfies
\begin{equation*}
v_{\epsilon}(x)=\widetilde{U}_{\delta_{1,\epsilon},\xi_{1,\epsilon}
}(x)-\widetilde{U}_{\delta_{2,\epsilon},\xi_{2,\epsilon}}(x)+o(1)\qquad
\text{in }D^{1,2}(\mathcal{D}),
\end{equation*}
with
\begin{equation*}
\epsilon^{-{\frac{n-1}{n-2}}}\delta_{i,\epsilon}\rightarrow d_{i}>0,\quad
\xi_{i,\epsilon}=\xi_{0}+\epsilon t_{i,\epsilon}\nu(\xi_{0})\quad
\hbox{and}\quad t_{i,\epsilon}\rightarrow t_{i}>0,
\end{equation*}
for each $i=1,2,$ as $\epsilon\rightarrow0.$
\end{theorem}

By the previous discussion Theorems \ref{main1} and \ref{main2} follow
immediately from Theorems \ref{main3} and \ref{main4}. The proof of Theorems
\ref{main3} and \ref{main4} relies on a very well known Ljapunov-Schmidt
reduction procedure. We shall omit many details on this procedure because they
can be found, up to some minor modifications, in the literature. We only
compute what cannot be deduced from known results.

The outline of the paper is as follows: In Section \ref{var-set} we write the
approximate solution, sketch the Ljapunov-Schmidt procedure and use it to
prove Theorems \ref{main3} and \ref{main4}. In Appendix B we compute the rate
of the error term and in Appendix C we estimate the reduced energy. In
Appendix A we give some important estimates on the Green function close to the boundary.

\section{The variational setting}

\label{var-set}We take
\begin{equation*}
(u,v):=\int_{\Omega}a(x)\nabla u\cdot\nabla v\text{ }dx,\text{\qquad
}\left\Vert u\right\Vert :=\left(  \int_{\Omega}a(x)\left\vert \nabla
u\right\vert ^{2}dx\right)  ^{1/2},
\end{equation*}
as the inner product in $\mathrm{H}_{0}^{1}(\Omega)$ and its corresponding
norm. Since $a$ is strictly positive and bounded in $\overline{\Omega}$ they
are well defined and equivalent to the standard ones. Similarly, for each
$r\in\lbrack1,\infty)$,
\begin{equation*}
\left\Vert u\right\Vert _{r}:=\left(  \int_{\Omega}a(x)\left\vert u\right\vert
^{r}dx\right)  ^{1/r}
\end{equation*}
is a norm in $\mathrm{L}^{r}(\Omega)$ which is equivalent to the standard one.

Next, we rewrite problem (\ref{p1}) in a different way. Let $i^{\ast
}:\mathrm{L}^{\frac{2n}{n+2}}(\Omega)\rightarrow\mathrm{H}_{0}^{1}(\Omega)$ be
the adjoint operator to the embedding $i:\mathrm{H}_{0}^{1}(\Omega
)\hookrightarrow\mathrm{L}^{\frac{2n}{n-2}}(\Omega),$ i.e. $i^{\ast}(u)=v$ if
and only if
\begin{equation*}
(v,\varphi)=\int_{\Omega}a(x)u(x)\varphi(x)dx\quad\text{for all }\varphi\in
C_{c}^{\infty}(\Omega)
\end{equation*}
if and only if
\begin{equation*}
-\text{ div}(a(x)\nabla v)=a(x)u\quad\text{in}\ \Omega,\qquad v=0\quad
\text{on}\ \partial\Omega.
\end{equation*}
Clearly, there exists a positive constant $c$ such that
\begin{equation*}
\left\Vert i^{\ast}(u)\right\Vert \leq c\left\Vert u\right\Vert _{\frac
{2n}{n+2}}\quad\forall\ u\in\mathrm{L}^{\frac{2n}{n+2}}(\Omega).
\end{equation*}
Setting $p :=\frac{2n}{n-2}$ and $f_{\epsilon}(s):=\left\vert s\right\vert ^{p
-2-\epsilon}s$, problem (\ref{p1}) turns out to be equivalent to
\begin{equation}
u=i^{\ast}\left(  f_{\epsilon}(u)\right)  ,\qquad u\in\mathrm{H}_{0}
^{1}(\Omega). \label{p2}
\end{equation}

Set $f(s):=f_{0}(s)$ and $\alpha_{n}:=\left[  n(n-2)\right]  ^{\frac{n-2}{4}
}.$ Let
\begin{equation*}
U_{\delta,\xi}:=\alpha_{n}{\frac{\delta^{\frac{n-2}{2}}}{(\delta^{2}
+|x-\xi|^{2})^{\frac{n-2}{2}}},}\qquad\delta>0,\text{\quad}\xi\in
\mathbb{R}^{n},
\end{equation*}
be the positive solutions to the limit problem
\begin{equation*}
-\Delta u=f(u),\qquad u\in H^{1}(\mathbb{R}^{n}).
\end{equation*}
Set
\begin{equation*}
\psi_{\delta,\xi}^{0}(x):={\frac{\partial U_{\delta,\xi}}{\partial\delta}
}=\alpha_{n}{\frac{n-2}{2}}\delta^{\frac{n-4}{2}}{\frac{|x-\xi|^{2}-\delta
^{2}}{(\delta^{2}+|x-\xi|^{2})^{n/2}}}
\end{equation*}
and, for each $j=1,\dots,n,$
\begin{equation*}
\psi_{\delta,\xi}^{j}(x):={\frac{\partial U_{\delta,\xi}}{\partial\xi_{j}}
}=\alpha_{n}(n-2)\delta^{\frac{n-2}{2}}{\frac{x_{j}-\xi_{j}}{(\delta
^{2}+|x-\xi|^{2})^{n/2}}}.
\end{equation*}
Recall that the space spanned by the $(n+1)$ functions $\psi_{\delta,\xi}^{j}
$ is the set of solutions to the linearized problem
\begin{equation*}
-\Delta\psi=(p -1)U_{\delta,\xi}^{p -2}\psi\text{\qquad in\ }\mathbb{R}^{n}.
\end{equation*}
Let $PW$ denote the projection of the function $W\in D^{1,2}(\mathbb{R}^{n})$
onto $\mathrm{H}_{0}^{1}(\Omega)$, i.e.
\begin{equation*}
\Delta PW=\Delta W\ \text{\ in}\ \Omega,\qquad PW=0\ \text{\ on}
\ \partial\Omega.
\end{equation*}

We look for two different types of solutions to problem (\ref{p1}). The
solutions found in Theorem \ref{main3} are of the form
\begin{equation}
u_{\epsilon}=\sum\limits_{i=1}^{\kappa}(-1)^{\lambda_{i}}PU_{\delta
_{i,\epsilon},\xi_{i,\epsilon}}+\phi, \label{ans1}
\end{equation}
for fixed $\lambda_{i}\in\{0,1\},$ where the concentration parameters satisfy
\begin{equation}
\delta_{i,\epsilon}=\epsilon^{\frac{n-1}{n-2}}d_{i}\quad\text{for some}\quad
d_{i}>0, \label{par}
\end{equation}
and the concentration points satisfy
\begin{equation}
\xi_{i,\epsilon}=s_{i}+\eta_{i}\nu(s_{i})\ \hbox{ where}\ s_{i}\in
\partial\Omega\ \hbox{and}\ \eta_{i}=\epsilon t_{i}\ \hbox{for some}\ t_{i}>0.
\label{pt1}
\end{equation}
Here and in the following $\nu(s_{i})$ denotes the inward unit normal to the
boundary $\partial\Omega$ at the point $s_{i}$.

On the other hand, the solutions found in Theorem \ref{main4} are of the form
\begin{equation}
u_{\epsilon}=\sum\limits_{i=1}^{\ell}(-1)^{i+1}PU_{\delta_{i,\epsilon},
\xi_{i,\epsilon}}+\phi, \label{ans2}
\end{equation}
where the concentration parameters satisfy (\ref{par}), while the
concentration points are aligned on the line $\mathfrak{L}:=\{\xi_{0}+r\nu
(\xi_{0}):r\in\mathbb{R}\},$ namely
\begin{equation}
\xi_{i,\epsilon}=\xi_{0}+\eta_{i}\nu(\xi_{0})\qquad\hbox{ where}\ \eta
_{i}=\epsilon t_{i}\ \hbox{for some}\ 0<t_{1}<\dots<t_{\ell}. \label{pt2}
\end{equation}

Next, we introduce the configuration space $\Lambda$ where concentration
parameters and concentration points lie. For solutions of type (\ref{ans1}) we
set $\mathbf{s}=(s_{1},\dots,s_{\kappa})\in(\partial\Omega)^{\kappa},$
$\mathbf{d}=(d_{1},\dots,d_{\kappa})\in(0,+\infty)^{\kappa},$ and
$\mathbf{t}=(t_{1},\dots,t_{\kappa})\in(0,+\infty)^{\kappa},$ and so
\begin{equation*}
\Lambda:=\left\{  (\mathbf{s},\mathbf{d},\mathbf{t})\in(\partial
\Omega)^{\kappa}\times(0,+\infty)^{\kappa}\times(0,+\infty)^{\kappa}
\ :\ s_{i}\not =s_{j}\ \hbox{if}\ i\not =j\right\}  ,
\end{equation*}
while for solutions of type (\ref{ans2}), we fix $\mathbf{s}=(\xi_{0}
,\dots,\xi_{0})$ and we set $\mathbf{d}=(d_{1},\dots,d_{\ell})\in
(0,+\infty)^{\ell},$ and $\mathbf{t}=(t_{1},\dots,t_{\ell})\in(0,+\infty
)^{\ell},$ and so
\begin{equation*}
\Lambda:=\left\{  (\mathbf{d},\mathbf{t})\in(0,+\infty)^{\ell}\times
(0,+\infty)^{\ell}\ :\ t_{1}<\dots<t_{\ell}\right\}  .
\end{equation*}
In each of these cases we write
\begin{equation*}
V_{\mathbf{s},\mathbf{d},\mathbf{t}}:=\sum\limits_{i=1}^{\kappa}
(-1)^{\lambda_{i}}PU_{\delta_{i},\xi_{i}}\text{\qquad and\qquad}
V_{\mathbf{s},\mathbf{d},\mathbf{t}}=V_{\mathbf{d},\mathbf{t}}:=\sum
\limits_{i=1}^{\ell}(-1)^{i+1}PU_{\delta_{i},\xi_{i}}
\end{equation*}
respectively.

The rest term $\phi$ belongs to a suitable space which we now define. For
simplicity we write $\psi_{i}^{j}:=\psi_{\delta_{i,\epsilon},\xi_{i,\epsilon}
}^{j}$ with $\delta_{i,\epsilon}$ as in \eqref{par} and $\xi_{i,\epsilon}$ as
in \eqref{pt1} or \eqref{pt2}.

For solutions of type \eqref{ans1} we introduce the spaces
\begin{equation*}
K_{\mathbf{s},\mathbf{d},\mathbf{t}}:=\mathrm{span}\{P\psi_{i}^{j}
\ :i=1,\dots,\kappa,\ j=0,1,\dots,n\},
\end{equation*}
\begin{equation*}
K_{\mathbf{s},\mathbf{d},\mathbf{t}}^{\perp}:=\left\{  \phi\in\mathrm{H}
_{0}^{1}(\Omega):(\phi,P\psi_{i}^{j})=0,\ i=1,\dots,\kappa,\ j=0,1,\dots
,n\right\}  .
\end{equation*}

Note that for $\xi_{i,\epsilon}$ as in \eqref{pt2} the functions $P\psi
_{i}^{j}$ are invariant with respect to the reflections $\varrho_{i}$ given in
$(a2).$ So for solutions of type \eqref{ans2} we define the space
$K_{\mathbf{s},\mathbf{d},\mathbf{t}}$ as above and $K_{\mathbf{s}
,\mathbf{d},\mathbf{t}}^{\perp}$ as the orthogonal complement of
$K_{\mathbf{s},\mathbf{d},\mathbf{t}}$ in the subspace of all functions in
$\mathrm{H}_{0}^{1}(\Omega)$ which are invariant with respect to $\varrho
_{1},\ldots,\varrho_{n-1}.$ Then we introduce the orthogonal projection
operators $\Pi_{\mathbf{s},\mathbf{d},\mathbf{t}}$ and $\Pi_{\mathbf{s}
,\mathbf{d},\mathbf{t}}^{\perp}$ in $H^{1}_{0}(\Omega)$ with ranges
$K_{\mathbf{s},\mathbf{d},\mathbf{t}}$ and $K_{\mathbf{s},\mathbf{d}
,\mathbf{t}}^{\bot}$, respectively.

As usual, our approach to solve problem \eqref{p2} will be to find a
$(\mathbf{s},\mathbf{d},\mathbf{t})\in\Lambda$ and a function $\phi\in
K_{\mathbf{s},\mathbf{d},\mathbf{t}}^{\perp}$ such that
\begin{equation}
\Pi_{\mathbf{s},\mathbf{d},\mathbf{t}}^{\perp}\left(  V_{\mathbf{s}
,\mathbf{d},\mathbf{t}}+\phi-i^{\ast}\left[  f_{\epsilon}(V_{\mathbf{s}
,\mathbf{d},\mathbf{t}}+\phi)\right]  \right)  =0 \label{es1}
\end{equation}
and
\begin{equation}
\Pi_{\mathbf{s},\mathbf{d},\mathbf{t}}\left(  V_{\mathbf{s},\mathbf{d}
,\mathbf{t}}+\phi-i^{\ast}\left[  f_{\epsilon}(V_{\mathbf{s},\mathbf{d}
,\mathbf{t}}+\phi)\right]  \right)  =0. \label{es2}
\end{equation}

First we shall find, for each $(\mathbf{s},\mathbf{d},\mathbf{t})\in\Lambda$
and small $\epsilon,$ a function $\phi\in K_{\mathbf{s},\mathbf{d},\mathbf{t}
}^{\perp}$ such that \eqref{es1} holds. To this aim we define a linear
operator $L_{\mathbf{s},\mathbf{d},\mathbf{t}}:K_{\mathbf{s},\mathbf{d}
,\mathbf{t}}^{\perp}\rightarrow K_{\mathbf{s},\mathbf{d},\mathbf{t}}^{\perp}$
by
\begin{equation*}
L_{\mathbf{s},\mathbf{d},\mathbf{t}}\phi:=\phi-\Pi_{\mathbf{s},\mathbf{d}
,\mathbf{t}}^{\perp}i^{\ast}\left[  f^{\prime}(V_{\mathbf{s},\mathbf{d}
,\mathbf{t}})\phi\right]  .
\end{equation*}
The following statement holds true.

\begin{proposition}
\label{pro1} For any compact subset $\mathbf{C}$ of $\Lambda$ there exist
$\epsilon_{0}>0$ and $c>0$ such that for each $\epsilon\in(0,\epsilon_{0}) $
and $(\mathbf{s},\mathbf{d},\mathbf{t})\in\mathbf{C}$ the operator
$L_{\mathbf{s},\mathbf{d},\mathbf{t}}$ is invertible and
\begin{equation*}
\left\Vert L_{\mathbf{s},\mathbf{d},\mathbf{t}}\phi\right\Vert \geq
c\left\Vert \phi\right\Vert \ \quad\ \forall\ \phi\in K_{\mathbf{s}
,\mathbf{d},\mathbf{t}}^{\perp}.
\end{equation*}

\end{proposition}

\begin{proof}
We argue as in Lemma 1.7 of \cite{MR1911045}.
\end{proof}

Now we are in position to solve equation \eqref{es1}.

\begin{proposition}
\label{pro2} For any compact subset $\mathbf{C}$ of $\Lambda$ there exist
$\epsilon_{0},c,\sigma>0$ such that for each $\epsilon\in
(0,\epsilon_{0}) $ and $(\mathbf{s},\mathbf{d},\mathbf{t})\in\mathbf{C}$ there
exists a unique $\phi_{\mathbf{s},\mathbf{d},\mathbf{t}}^{\epsilon}\in
K_{\mathbf{s},\mathbf{d},\mathbf{t}}^{\perp}$ such that \eqref{es1} holds and
\begin{equation}
\left\Vert \phi_{\mathbf{s},\mathbf{d},\mathbf{t}}^{\epsilon}\right\Vert \leq
c\epsilon^{{\frac{1}{2}}+\sigma} . \label{eq2-pro2}
\end{equation}

\end{proposition}

\begin{proof}
We estimate the rate of the error term
\begin{equation}
R_{\mathbf{s},\mathbf{d},\mathbf{t}}:=\Pi_{\mathbf{s},\mathbf{d},\mathbf{t}
}^{\perp}\left(  V_{\mathbf{s},\mathbf{d},\mathbf{t}}-i^{\ast}\left[
f_{\epsilon}(V_{\mathbf{s},\mathbf{d},\mathbf{t}})\right]  \right)  \label{re}
\end{equation}
in Appendix B. Then we argue exactly as in Proposition 2.3 of \cite{MR2232205}.
\end{proof}

The critical points of the energy functional $J_{\epsilon}:\mathrm{H}_{0}
^{1}(\Omega)\rightarrow\mathbb{R}$ defined by
\begin{equation*}
J_{\epsilon}(u):={\frac{1}{2}}\int\limits_{\Omega}a(x)|\nabla u|^{2}
dx-{\frac{1}{p -\epsilon}}\int\limits_{\Omega}a(x)|u|^{p -\epsilon}dx
\end{equation*}
are the solutions to problem \eqref{p1}. We define the reduced energy
functional $\widetilde{J}_{\epsilon}:\Lambda\rightarrow\mathbb{R}$ by
\begin{equation*}
\widetilde{J}_{\epsilon}(\mathbf{s},\mathbf{d},\mathbf{t}):=J_{\epsilon
}(V_{\mathbf{s},\mathbf{d},\mathbf{t}}+\phi_{\mathbf{s},\mathbf{d},\mathbf{t}
}^{\epsilon})
\end{equation*}
The critical points of $\widetilde{J}_{\epsilon}$ are the solutions to problem \eqref{es2}.

\begin{proposition}
\label{pro3} The function $V_{\mathbf{s},\mathbf{d},\mathbf{t}}+\phi
_{\mathbf{s},\mathbf{d},\mathbf{t}}^{\epsilon}$ is a critical point of the
functional $J_{\epsilon}$ if and only if the point $(\mathbf{s},\mathbf{d}
,\mathbf{t})$ is a critical point of the function $\widetilde{J}_{\epsilon}.$
\end{proposition}

\begin{proof}
We argue as in Proposition 1 of \cite{MR1384837}.
\end{proof}

The problem is thus reduced to the search for critical points of
$\widetilde{J}_{\epsilon},$ so it is necessary to compute the asymptotic
expansion of $\widetilde{J}_{\epsilon}$.

\begin{proposition}
\label{pro4-1}In case \eqref{ans1} it holds true that
\begin{align}
&  \widetilde{J}_{\epsilon}(\mathbf{s},\mathbf{d},\mathbf{t})=\left(
c_{1}+c_{2}\epsilon\log\epsilon\right)
{\sum\limits_{i=1}^{\kappa}}
a(s_{i})\label{eq1-pro4-1}\\
&  +\epsilon
{\sum\limits_{i=1}^{\kappa}}
\left[  c_{3}a(s_{i})+c_{4}\langle\nabla a(s_{i}),\nu(s_{i})\rangle
t_{i}+c_{5}a(s_{i})\left(  \frac{d_{i}}{2t_{i}}\right)  ^{n-2}-c_{6}
a(s_{i})\log d_{i}\right]  +o(\epsilon),\nonumber
\end{align}
$C^{1}$-uniformly on compact sets of $\Lambda.$ Here the $c_{i}$'s are
constants and $c_{4},c_{5},c_{6}$ are positive.
\end{proposition}

\begin{proof}
The proof is postponed to Appendix C.
\end{proof}

\begin{proposition}
\label{pro4}In case \eqref{ans2} it holds true that
\begin{equation}
\widetilde{J}_{\epsilon}(\mathbf{s},\mathbf{d},\mathbf{t})=\widetilde
{J}_{\epsilon}(\mathbf{d},\mathbf{t})=a(\xi_{0})\left[  c_{1}+c_{2}
\epsilon\log\epsilon+c_{3}\epsilon\right]  +\epsilon\Psi(\mathbf{d}
,\mathbf{t})+o(\epsilon), \label{eq1-pro4}
\end{equation}
$C^{0}$-uniformly on compact sets of $\Lambda.$ Here
\begin{multline}\label{psi}
  \Psi(\mathbf{d},\mathbf{t})
  :=c_{4}\langle\nabla a(\xi_{0}),\nu(\xi_{0})\rangle
  {\sum\limits_{i=1}^{\ell}}
  t_{i}
  +c_{5}a(\xi_{0})\times\\
  \times\Biggl\{
  {\sum\limits_{i=1}^{\ell}}
  \left(  \frac{d_{i}}{2t_{i}}\right)  ^{n-2}
  +{\sum\limits_{\genfrac{}{}{0pt}{}{i,j=1}{i\not =j}}^{\ell}}
  (-1)^{i+j+1}(d_{i}d_{j})^{\frac{n-2}{2}}\left[  \frac{1}{|t_{i}-t_{j}|^{n-2}
    }-\frac{1}{|t_{i}+t_{j}|^{n-2}}\right]  \Biggr\} \\
  \phantom{:={}}   -c_{6}a(\xi_{0})
  {\sum\limits_{i=1}^{\ell}}
  \log d_{i}
\end{multline}
where the $c_{i}$'s are constants and $c_{4},c_{5},c_{6}$ are positive.
\end{proposition}

\begin{proof}
The proof is postponed to Appendix C.
\end{proof}

\begin{proof}
[Proof of Theorem \ref{main3}]Firstly, by Proposition \ref{pro4-1}, we get
\begin{equation*}
\widetilde{J}_{\epsilon}(\mathbf{s},\mathbf{d},\mathbf{t})=\left(  c_{1}
+c_{2}\epsilon\log\epsilon\right)
{\sum\limits_{i=1}^{\kappa}}
a(s_{i})+O(\epsilon),
\end{equation*}
$C^{1}$-uniformly on compact sets of $\Lambda.$ Then, since $\xi_{1},\dots
,\xi_{\kappa}$ are non degenerate critical points of $a$ constrained to the
boundary of $\Omega,$ if $\epsilon$ is small enough there exist $\mathbf{s}
_{\epsilon}:=(s_{1,\epsilon},\dots,s_{\kappa,\epsilon})$ such that each
$s_{i,\epsilon}\rightarrow\xi_{i}$ as $\epsilon$ goes to zero, and
$\nabla_{\mathbf{s}}\widetilde{J}_{\epsilon}(\mathbf{s}_{\epsilon}
,\mathbf{d},\mathbf{t})=0.$ Secondly, by Proposition \ref{pro4-1}, we also
get
\begin{multline*}
\widetilde{J}_{\epsilon}(\mathbf{s}_{\epsilon},\mathbf{d},\mathbf{t}
)-\left(  c_{1}+c_{2}\epsilon\log\epsilon\right)
{\sum\limits_{i=1}^{\kappa}}
a(s_{i,\epsilon})\\
\begin{aligned}
&  =\epsilon\sum\limits_{i=1}^{\kappa}\Biggl[  c_{3}a(s_{i,\epsilon}
)+c_{4}\langle\nabla a(s_{i,\epsilon}),\nu(s_{i,\epsilon})\rangle t_{i}\\
&\hspace{4cm}+c_{5}a(s_{i,\epsilon})\left(  \frac{d_{i}}{2t_{i}}\right)  ^{n-2}
-c_{6}a(s_{i,\epsilon})\log d_{i}\Biggr]  +o(\epsilon)\\
&  =\epsilon\sum\limits_{i=1}^{\kappa}\Biggl[  c_{3}a(\xi_{i})+c_{4}
\langle\nabla a(\xi_{i}),\nu(\xi_{i})\rangle t_{i}\\
&\hspace{4cm}+c_{5}a(\xi_{i})\left(
\frac{d_{i}}{2t_{i}}\right)  ^{n-2}-c_{6}a(\xi_{i})\log d_{i}\Biggr]
+o(\epsilon).
\end{aligned}
\end{multline*}
It is easy to verify that the function
\begin{equation*}
(\mathbf{d},\mathbf{t})\rightarrow\sum\limits_{i=1}^{\kappa}\left[
c_{4}\langle\nabla a(\xi_{i}),\nu(\xi_{i})\rangle t_{i}+c_{5}a(\xi_{i})\left(
\frac{d_{i}}{2t_{i}}\right)  ^{n-2}-c_{6}a(\xi_{i})\log d_{i}\right]
\end{equation*}
has a minimum point which is stable under $C^{0}$-perturbations. Therefore,
there exists a point $(\mathbf{d}_{\epsilon},\mathbf{t}_{\epsilon})$ such that
$\nabla_{(\mathbf{d},\mathbf{t})}\widetilde{J}_{\epsilon}(\mathbf{s}
_{\epsilon},\mathbf{d}_{\epsilon},\mathbf{t}_{\epsilon})=0.$ Thus, the
function $\widetilde{J}_{\epsilon}$ has a critical point and the claim follows
from Proposition \ref{pro3}.
\end{proof}

\begin{proof}
[Proof of Theorem \ref{main4}]In this case $\ell=2$ and function $\Psi$
defined in \eqref{psi} reduces to
\begin{align*}
&  \Psi(\mathbf{d},\mathbf{t})=c_{4}\langle a(\xi_{0}),\nu(\xi_{0}
)\rangle(t_{1}+t_{2})\\
&  +c_{5}a(\xi_{0})\left\{  \left(  \frac{d_{1}}{2t_{1}}\right)
^{n-2}+\left(  \frac{d_{2}}{2t_{2}}\right)  ^{n-2}+2(d_{1}d_{2})^{\frac
{n-2}{2}}\left[  \frac{1}{|t_{i}-t_{j}|^{n-2}}-\frac{1}{|t_{i}+t_{j}|^{n-2}
}\right]  \right\} \\
&  -c_{6}a(\xi_{0})(\log d_{1}+\log d_{2}).
\end{align*}
It is easy to verify that it has minimum point which is stable under $C^{0}
$-perturbations. Therefore, from Proposition \ref{pro4} we deduce that, if
$\epsilon$ is small enough, the function $\widetilde{J}_{\epsilon}$ has a
critical point. Now the claim follows from Proposition \ref{pro3}.
\end{proof}

\begin{remark}
\label{rem} The symmetry assumption $(a2)$ allows to overcome some
technical difficulties which arise when looking for a solution whose bubbles
collapse to the same point. Indeed, the problem arises when we study the
reduced energy and we have to compute the contribution of each peak and the
interaction among the peaks. The contribution of each peak is clear: it is
given by the distance from the peak to the boundary as in \eqref{eq:10} and by
the value of the function $a$ at the projection of the peak onto the
boundary as in \eqref{eq:100}. On the other hand, to compute the interaction
among the peaks (see \eqref{eq:11}) it is important to compare the geodesic
distance $d(s_{i},s_{j})$ between the projections of the peaks onto the
boundary with the distance $|\eta_{i}\nu(s_{i})-\eta_{j}\nu(s_{j})|$
between the normal components of the peaks. To have a good expansion the
distance $d(s_{i},s_{j})$ should be negligible with respect to the
distance $|\eta_{i}\nu(s_{i})-\eta_{j}\nu(s_{j})|.$ But then, in order
to find a criticality in the points $s_{i}$, we need to go further in
the expansion and computations become too tedious. If the domain $\Omega
$ and the function $a$ are symmetric, we can overcome this
difficulty just by assuming that the peaks satisfy \eqref{pt2}, so that
$d(s_{i},s_{j})=0.$ In this case the interaction among the peaks is
clear and it is given in terms of the Green function of the Laplace operator
on the half-space (see \eqref{eq:11}). 
\end{remark}

\appendix

\section{Boundary estimates of the Green function}

In this section we establish the technical estimates we used in the previous
part. We denote by $G(x,y)$ the Green function of the Laplacian with Dirichlet
boundary condition and by $H(x,y)$ its regular part, i.e.
\begin{equation*}
G(x,y)=\frac{1}{n(n-2)\omega_{n}|x-y|^{n-2}}-H(x,y),
\end{equation*}
where $\omega_{n}$ is the volume of the unit ball in $\mathbb{R}^{n}$.

First of all, we need an accurate estimate of $H(x,y)$ when the points $x$ and
$y$ are close to the boundary. Let us introduce some notation. For $\eta>0$ we
write $\Omega_{\eta}:=\{x\in\Omega\ :\ \mathrm{dist}(x,\partial\Omega)\leq
\eta\}$. We fix $\eta$ small enough so that the orthogonal projection
$p:\Omega_{2\eta}\rightarrow\partial\Omega$ onto the boundary is well defined,
i.e. so that for each $x\in\Omega_{2\eta}$ there is a unique point
$p(x)\in\partial\Omega$ with $\mathrm{dist}(x,\partial\Omega)=|p(x)-x|$. Set
$d_{x}:=\mathrm{dist}(x,\partial\Omega)$, $p_{x}:=p(x)$, and $\nu_{x}:=\nu
(x)$, where as before $\nu(x)$ denotes the inward normal to $\partial\Omega$
at $x$. For $x\in\Omega_{2\eta}$ we define $\bar{x}:=p_{x}-d_{x}\nu
_{x}=x-2d_{x}\nu_{x}.$ Thus, $\bar{x}$ is the reflection of $x$ on
$\partial\Omega$.

 \begin{lemma}\label{lem1}
   There exists $C>0$ such that
   \begin{align}\label{eq:2}
     \xabs{H(x,y)-\frac{1}{\abs{\bar x-y}^{n-2}}}&\le\frac{Cd_x}{\abs{\bar x-y}^{n-2}}\\
     \xabs{\nabla_x\xlr(){H(x,y)-\frac{1}{\abs{\bar x-y}^{n-2}}}}&\le\frac{C}{\abs{\bar
         x-y}^{n-2}}\label{eq:3}
   \end{align}
   for all $x\in\Omega_\eta$ and $y\in\Omega$.  In particular, there
   exists $C>0$ such that
   \begin{equation}\label{eq:31}
     0\le H(x,y) \le \frac{C}{ \abs{\bar x-  y}^{n-2}},
     \qquad x\in\Omega_\eta,\ y\in\Omega 
   \end{equation}
   and
  \begin{equation}
    \label{eq:15}
    \abs{\nabla_x H(x,y)}\le \frac{C}{\abs{x-y}^{n-1}}
    \qquad x,y\in\Omega.
  \end{equation}
 \end{lemma}
\begin{proof}
  For convenience we set
  \begin{equation*}
    \chi(x,y)\coloneqq H(x,y)-\frac{1}{\abs{\bar x-y}^{n-2}}
  \end{equation*}
  for $x\in\Omega_\eta$ and $y\in\Omega$.  Note that there is
  $c>0$, only dependent on $n$ and $\eta$, such that $\abs{\bar
    x-\bar\xi}\le c\abs{x-\xi}$ if $x\in\Omega_\eta$ and $\xi\in
  \olB(x,d_x/2)$. If moreover $y\in\Omega$, then
  \begin{equation}
    \label{eq:4}
    \frac{\abs{\bar x-y}}{\abs{\bar\xi-y}}
    \le\frac{\abs{\bar x-\bar\xi}+\abs{\bar\xi-y}}{\abs{\bar\xi-y}}
    \le1+\frac{cd_x/2}{\abs{\bar\xi-y}}
    \le 1+c,
  \end{equation}
  since $y\in\Omega$ and $\dist(\bar\xi,\Omega)\ge d_x/2$.
  
  The proof of \eqref{eq:2} is analogous to the proof of Eq.~(2.7) in
  \cite{MR2653899}, with obvious small changes.  Similarly, slight
  modifications of the proof of Eq.~(2.8) in \cite{MR2653899} yield
  \begin{equation}
    \label{eq:1}
    \abs{\Delta_x\chi(x,y)}\le\frac{C}{d_x\abs{\bar x-y}^{n-2}}
  \end{equation}
  for all $x\in\Omega_\eta$ and $y\in\Omega$.  Fix $x,y$, take
  $r\coloneqq2\sqrt{n}$ and set
  \begin{equation*}
    Q\coloneqq\{\xi\in\dR^n\mid\abs{x-\xi}_\infty \le d_x/r\}.
  \end{equation*}
  Note that if $\xi\in Q$ then $\xi\in\olB(x,d_x/2)$ and
  therefore
  \begin{equation}
    d_x/2 \le d_\xi\le3d_x/2.\label{eq:5}
  \end{equation}
  Hence we obtain for $i\in\{1,2,\dots,n\}$
  \begin{align*}
    \abs{\partial_{x_i}\chi(x,y)}
    &\le\frac{rn}{d_x}\sup_{\xi\in\partial
      Q}\abs{\chi(\xi,y)}+\frac{d_x}{2r}\sup_{\xi\in
      Q}\abs{\Delta_\xi\chi(\xi,y)}
    &&\qquad\text{by \cite[Eq.~(3.15)]{MR737190}}\\
    &\le C\xlr(){\sup_{\xi\in \partial
        Q}\frac{d_\xi}{d_x\abs{\bar\xi-y}^{n-2}}
    +\sup_{\xi\in Q}\frac{d_x}{d_\xi\abs{\bar\xi-y}^{n-2}}}
    &&\qquad\text{by \eqref{eq:2} and \eqref{eq:1}}\\
    &\le C\sup_{\xi\in Q}\frac{1}{\abs{\bar\xi-y}^{n-2}}
    &&\qquad\text{by \eqref{eq:5}}\\
    &\le \frac{C}{\abs{\bar x-y}^{n-2}}
    &&\qquad\text{by \eqref{eq:4}}.
  \end{align*}
  Summing up this inequality over $i$ gives \eqref{eq:3}.

  To prove \eqref{eq:15}, note first that there is $C>0$ such that
  \begin{equation}
    \label{eq:14}
   \abs{\nabla_xH(x,y)}\le C 
   \qquad\text{if }
   x\in\Omega\ssm\Omega_\eta,\ y\in\Omega.
  \end{equation}
  The case $x\in\Omega_\eta$ relies on the estimate \eqref{eq:3}.
  Note that there is $C>0$ such that
  \begin{equation}\label{eq:8}
    \frac{\abs{\bar x-y}}{\abs{x-y}}\ge C
    \qquad\text{for all } x\in\Omega_\eta,\ y\in\Omega.
  \end{equation}
  This implies that the term on the right of \eqref{eq:3} is estimated
  by a constant multiple of $1/\abs{x-y}^{n-2}$ if $x\in\Omega_\eta$
  and $y\in\Omega$.  In view of \eqref{eq:14} it therefore remains to
  show that
  \begin{equation}
    \label{eq:21}
    \xabs{\nabla_x\frac{1}{\abs{\bar x-y}^{n-2}}}
    \le
    \frac{C}{\abs{x-y}^{n-1}}
    \qquad x\in\Omega_\eta,\ y\in\Omega
  \end{equation}
  for some constant $C>0$.

  Writing $\partial_i$ for $\partial/\partial x_i$ we calculate as in
  \cite{MR2653899} for any $i\in\{1,2,\dots,n\}$:
  \begin{equation}\label{eq:9}
    \partial_i\frac1{\abs{\bar x-y}^{n-2}}
    =\frac{2-n}{\abs{\bar x-y}^n}\sum_{j=1}^n(\bar{x}_j-y_j)\partial_i\bar{x}_j.
  \end{equation}
  Since $\bar x:=x-2d_x \nu_x $, we find
  \begin{equation*}
    \partial_i\bar x_j=\delta_{ij}-2\nu_{xi}\nu_{xj}-2d_x \partial_i\nu_{xj}.
  \end{equation*}
  Using this representation in \eqref{eq:9} yields
  \begin{equation*}
    \xabs{\partial_i\frac1{\abs{\bar x-y}^{n-2}}}
    \le\frac{C}{\abs{\bar x-y}^{n-1}}(1+d_x \abs{\partial_i\nu_x }).
  \end{equation*}
  By our choice of $\eta$ we have $\abs{d_x }\le\eta$ and
  $\abs{\partial_i\nu_x }\le C$ for all $x\in\Omega_\eta$.  In view of
  \eqref{eq:8} we obtain \eqref{eq:21} and finish the proof.
\end{proof}

Here and in the remaining appendices we employ the notation
\begin{equation*}
  \abs{u}_{A,q}\coloneqq\xlr(){\int_A\abs{u}^q}^{1/q}
\end{equation*}
for measurable $A\subseteq\dR^n$ and $q\in[1,\infty]$.  If $A=\Omega$
we omit it from the notation.

\begin{lemma}\label{lem2}
  Let $\delta,\delta_1,\delta_2\in(0,1]$ and $\xi,\xi_1,\xi_2\in\Omega_\eta$.  Let
  $\bar\xi$ be the reflection point of $\xi$ with respect to
  $\partial\Omega.$ There exists $c>0$ such that
  \begin{equation}
    \label{eq:12}
        0\le PU_{\delta,\xi}(x) \le U_{\delta,\xi}(x)\\
  \end{equation}
  and
  \begin{equation}\label{cru1}
    0\le U_{\delta,\xi}(x)-PU_{\delta,\xi}(x) 
    \le \alpha_n\delta^{\frac{n-2}{2}}H(x,\xi)
    \le c \frac{\delta^{\frac{n-2}{2}}}{|x-\bar \xi|^{n- 2}}
  \end{equation}
  for all $x\in\Omega$.  Moreover
  \begin{equation*}
    R_{\delta,\xi}(x)
    :=PU_{\delta,\xi}(x)-U_{\delta,\xi}(x)+ \alpha_n\delta^{\frac{n-2}{2}}H(x,\xi)
  \end{equation*}
  satisfies
  \begin{equation}\label{cru3}
    \abs{R_{\delta,\xi}}_{\Omega,\infty}
    =O\fracwithdelims(){\delta^{\frac{n+2}{2}}}{\dist(\xi,\partial\Omega)^n}.
  \end{equation}
  Finally, there is $\beta>0$ such that
  \begin{align}\label{cru2}
    \int_\Omega\abs{\nabla PU_{\delta_1,\xi_1}}PU_{\delta_2,\xi_2}
      &=\fracwithdelims(){\delta_1}{\delta_2}^{\frac{n-2}{2}}
      O\xlr(){\delta_2^{\frac{n-2}{n-1}+\beta}}\\
    \abs{\nabla PU_{\delta,\xi}}_{\frac{2n}{n+2}}
    &=O\biglr(){\delta^{\frac{n-2}{2(n-1)}+\beta}}\label{eq:16}
  \end{align}
  as $\delta,\delta_1,\delta_2\to0$, independently of $\xi$, $\xi_1$, and $\xi_2$.
\end{lemma}

\begin{proof}
  Estimates \eqref{eq:12}, \eqref{cru1}, and \eqref{cru3} follow
  easily from the maximum principle and Lemma \ref{lem1}.

  Note first that 
  \begin{align}
    \label{eq:13}
    \abs{U_{\delta,\xi}}_q
    &=O\xlr(){\delta^{\frac{n}{q}-\frac{n-2}{2}}}
    &&\qquad\text{if }q>\frac{n}{n-2}\\
    \intertext{and}
    \label{eq:7}
    \abs{U_{\delta,\xi}^{\frac{n+2}{n-2}}}_{q}
    &=O\xlr(){\delta^{\frac{n}{q}-\frac{n+2}{2}}}
    &&\qquad\text{if } q\ge1,
  \end{align}
  as $\delta\to0$, independently of $\xi$.

  Recall that
  \begin{equation}\label{eq:6}
      \nabla PU_{\delta,\xi}(x)
      =\int_\Omega\nabla_x\xlr(){\frac{1}{n(n-2)\omega_n |x-y|^{n-2}}
        -H(x,y)}U_{\delta,\xi}^{\frac{n+2}{n-2}}(y)\dint y
  \end{equation}
  and note that
  \begin{equation}
    \xabs{\nabla_x \frac{1}{|x-y|^{n-2}}}\le \frac{C}{\abs{x-y}^{n-1}}\label{eq:17}.
  \end{equation}
  By \eqref{eq:6}, \eqref{eq:15}, and \eqref{eq:17}, to show
  \eqref{cru2} it suffices to prove
  \begin{equation}
    \label{eq:27}
    \int_\Omega\int_\Omega
    U_{\delta_2,\xi_2}(x)\frac{1}{\abs{x-y}^{n-1}}
    U_{\delta_1,\xi_1}^{\frac{n+2}{n-2}}(y)
    \dint y\dint x
    =\fracwithdelims(){\delta_1}{\delta_2}^{\frac{n-2}{2}}
    O\xlr(){\delta_2^{\frac{n-2}{n-1}+\beta}}.
  \end{equation}
  For simplicity, set $V\coloneqq U_{\delta_1,\xi_1}^{\frac{n+2}{n-2}}$
  and $g(x)\coloneqq1/\abs{x}^{n-1}$.  Set
  $M\coloneqq\diam(\Omega)$.  Pick
  \begin{equation*}
    r\in\xlr(){\frac{n(n-1)}{(n-1)^2+1},\ \frac{n}{n-1}}
  \end{equation*}
  and note that then $r\ge1$ and $r'>n$, where $r'$ denotes the
  conjugate exponent of $r$.  Since $\frac{1}{r'}+\frac{1}{r}+1=2$ it
  follows as in the proof of \cite[Theorem~4.2]{MR1817225} that 
  \begin{multline*}
    \int_\Omega\int_\Omega
    U_{\delta_2,\xi_2}(x)g(x-y)
    V(y)
    \dint y\dint x
    \le\abs{U_{\delta_2,\xi_2}}_{r'}\abs{g}_{B(0,M),r}\abs{V}_1\\
    =O\xlr(){\delta_2^{\frac{n}{r'}-\frac{n-2}{2}}\delta_1^{n-\frac{n+2}{2}}}
    =\fracwithdelims(){\delta_1}{\delta_2}^{\frac{n-2}{2}}
    O\xlr(){\delta_2^{n\xlr(){1-\frac1r}}},
  \end{multline*}
  by \eqref{eq:13} and \eqref{eq:7}.  Here we have used that
  $\abs{g}_{B(0,M),r}$ is finite since $r<n/(n-1)$.  On the other
  hand, $r>n(n-1)/((n-1)^2+1)$ implies that
  \begin{equation*}
    n\xlr(){1-\frac1r}=\frac{n-2}{n-1}+\beta
  \end{equation*}
  for some $\beta>0$, proving \eqref{eq:27} and hence \eqref{cru2}.

  To prove \eqref{eq:16} we proceed similarly.  This time we pick
  \begin{equation*}
    s\in\xlr(){\max\xlr\{\}{1,\ \frac{2n}{n+4}},\ \frac{2n(n-1)}{n^2+2n-4}}
  \end{equation*}
  and define $r$ by
  \begin{equation}
    \label{eq:29}
    \frac1r+\frac1s=1+\frac{n+2}{2n}.
  \end{equation}
  Some basic calculations reveal that $s$ is well defined and that
  \begin{equation}
    \label{eq:30}
    r\in\xlr[){1,\ \frac{n}{n-1}}.
  \end{equation}
  Similarly to the proof of \cite[Theorem~4.2]{MR1817225}, taking
  into account the Remark~(2) following the statement of that theorem,
  we obtain
  \begin{equation*}
   \abs{\nabla PU_{\delta,\xi}}_{\frac{2n}{n+2}}
   \le\abs{g}_{B(0,M),r}\abs{V}_s
    =O\biglr(){\delta^{\frac{n}{s}-\frac{n+2}{2}}} .
  \end{equation*}
  Again we have used that $r<n/(n-1)$ implies that the $r$-norm of $g$
  in the ball of radius $M$ is finite.  Since $s<2n(n-1)/(n^2+2n-4)$,
  there is $\beta>0$ such that
  \begin{equation*}
    \frac{n}{s}-\frac{n+2}{2}=\frac{n-2}{2(n-1)}+\beta,
  \end{equation*}
  proving \eqref{eq:16}.
\end{proof}

\section{An estimate of the error}

To simplify notation, from now on we write
\begin{equation*}
\delta_{i}:=\delta_{i,\epsilon},\qquad\xi_{i}:=\xi_{i,\epsilon},\qquad
U_{i}:=U_{\delta_{i},\xi_{i}}.
\end{equation*}
Next, we estimate the error term defined in \eqref{re}.

\begin{lemma}\label{re1}
It holds true for some $\sigma>0$ that
 \begin{equation*}
    \norm{R_{\si,\di,\ti}} =O\left(\eps^{{\frac{1}{2}}+\sigma}\right).
\end{equation*}
\end{lemma}
\begin{proof}
  We estimate $ R_{ \si,\di,\ti}$ in case \eqref{ans2}.  The estimate
  in case \eqref{ans1} is easier and can be obtained after minor
  modifications of this argument.

 From the definition of $i^*$ we deduce that
\begin{equation}\label{re1.1}
  \begin{aligned}
    \norm{R_{\si,\di,\ti}}
    &=O\left(\abs{-\text{div}\left(a(x)\nabla V_{ \si,\di,\ti}\right)-a(x)f_\eps\left(V_{ \si,\di,\ti}\right)}_{\frac{2n}{ n+2}}\right) \\
    &=O\left(\abs{-\nabla a \nabla V_{\si, \di,\ti} -a(x)\Delta V_{ \si,\di,\ti}-a(x)f_\eps\left(V_{\si, \di,\ti}\right)}_{\frac{2n}{ n+2}}\right) \\
    &=O\left( \sum\limits_i\abs{\nabla a \nabla PU_i}_{\frac{2n}{ n+2}}\right) +O\left(\sum\limits_i\abs{a(x)\left[f(U_i)-f(PU_i)\right]}_{\frac{2n}{ n+2}}\right) \\
    &\hspace{2cm}+ O\left(\abs{a(x) [\sum\limits_i f(PU_i)-f(\sum\limits_i PU_i)
        ]}_{\frac{2n}{ n+2}}\right)\\
    &\hspace{2cm}+O\left(\abs{
        a(x)\left[f \left(V_{\si, \di,\ti}\right)-f_\eps\left(V_{\si, \di,\ti}\right)\right]}_{\frac{2n}{ n+2}}\right) \\
    &=:I_1+I_2+I_3+I_4.
  \end{aligned}
\end{equation}

To estimate $I_1$ recall that
$\delta_i=O\xlr(){\epsilon^{\frac{n-1}{n-2}}}$ on compact subsets of
$\Lambda$.  By \eqref{eq:16} we get, for some $\sigma>0$,
\begin{equation}\label{re1.2}
  \abs{\nabla a \nabla PU_i}_{\frac{2n}{ n+2}}
  =O\xlr(){\epsilon^{\frac12+\sigma}}.
\end{equation}

Let us estimate $I_2.$
By \eqref{cru1} for some $\sigma>0$ we obtain
\begin{multline}\label{re1.5}
  \abs{a \left[f(U_i)-f(PU_i)\right]}_{\frac{2n}{ n+2}}\\
  =O\left(\abs{U_i^{\pc-2}(PU_i-U_i)}_{\frac{2n}{ n+2}}\right)+O\left(\abs{|PU_i-U_i|^{\pc-1}}_{\frac{2n}{ n+2}}\right)
   =O\left(\eps^{{\frac{1}{2}}+\sigma}\right),
\end{multline}
because by \eqref{cru1} (using also \eqref{re1.7.2} with $q=(n+2)/4$)
\begin{multline}\label{re1.6}
  \abs{|PU_i-U_i|^{\pc-1}}_{\frac{2n}{ n+2}}
  =\abs{PU_i-U_i}^{\pc-1}_{\frac{2n}{ n-2}}\\
  =\delta_i^{\frac{n+2}{ 2}}O\left(\xabs{{\frac{1}{ |x-\bar\xi_i|^{n-2}}}}^{\pc-1}_{\frac{2n}{ n-2}}\right)
  =O\left(\delta_i^{\frac{n+2}{ 2}}\eps^{-{\frac{n+2}{ 2}}}\right)\end{multline}
and by H\"older's inequality  for some $\sigma>0$  (using also \eqref{re1.7.1} and \eqref{re1.7.2} with $q\sim1$ when $n\le6$ or $q\sim(n+2)/8$ when $n\ge7$)
\begin{equation}\label{re1.7}
  \begin{aligned}
    \abs{U_i^{\pc-2}(PU_i-U_i)}_{\frac{2n}{ n+2}}
    &=\delta_i^{\frac{n-2}{2}} O\left(\abs{U_i}^{\pc-2}_{\frac{8nq}{
          (n-2)(n+2)}}\right)O\left(\xabs{{\frac{1}{
            |x-\bar\xi_i|^{n-2}}}} _{\frac{2nq}{ (q-1)(n+2)}}\right)\\
    &=
    \begin{cases}
      \displaystyle O\left( \left({\frac{\delta_i}{ \epsilon}}\right) ^{
          {\frac{n+2}{2}}-\sigma}\right)
      &\qquad\text{if } n\ge7\\[3ex]
        \displaystyle O\left(\left({\frac{\delta_i}{ \epsilon}}\right) ^{ {n-2}-\sigma} \right)
        &\qquad\text{if } n\le6,
    \end{cases}
  \end{aligned}
\end{equation}
with
\begin{equation}\label{re1.7.1}\abs{U_i}^{\pc-2}_{\frac{8nq}{
      (n-2)(n+2)}}
  =
  \begin{cases}
    O\left(  \delta_i ^2\right)&\qquad\text{if } n\ge7 \text{ and }
      1<q<{\frac{n+2}{8}},\\[2ex]
      O\left(\delta_i ^{{\frac{n+2}{2q}}-2}\right)&\qquad\text{if } n\le6\text{ and } q>{\frac{n+2}{8}}.
\end{cases}
\end{equation}
and
\begin{multline}\label{re1.7.2}\xabs{{\frac{1}{ |x-\bar\xi_i|^{n-2}}}}
  _{\frac{2nq}{ (q-1)(n+2)}}
  =O\left( \eps^{-{\frac{n-6}{2}}-{\frac{n+2}{2q}}}\right)\\
    \text{if } n\ge6\text{ and } q>1\text{ or } n\le5\text{ and
      } 1<q<{\frac{n+2}{6-n}}.
\end{multline}

Let us estimate $I_3.$
We set
\begin{equation}
\label{eta}\eta:=\min\left\{{\rm d}(\xi_1,\partial\Omega),{\rm d}(\xi_2,\partial\Omega),{\frac{|\xi_1-\xi_2|}{2}}\right\} .\end{equation}
We have
\begin{multline}\label{re1.8}
    \biggabs{a(x) \bigglr[]{\sum\limits_i f(PU_i)-f\bigglr(){\sum\limits_i PU_i} }}_{\frac{2n}{ n+2}}\\
  \begin{aligned}
    &
    =O\left(\biggabs{\sum\limits_i f(PU_i)-f\bigglr(){\sum\limits_i PU_i}}_{\Omega\setminus \cup_iB(\xi_i,\eta),{\frac{2n}{ n+2}}}\right)\\
    &
    \hspace{1cm}+O\left(\sum_i\biggabs{f(PU_i)-f\bigglr(){\sum\limits_i PU_i}}_{ B(\xi_i,\eta),{\frac{2n}{ n+2}}}\right)\\
    &\hspace{1cm}+O\bigglr(){\sum_i\sum_{j\not=i}\abs{f(PU_j)}_{
        B(\xi_i,\eta),{\frac{2n}{ n+2}}}},
  \end{aligned}
\end{multline}
because
\begin{multline}\label{re1.9}
  \biggabs{\sum\limits_i f(PU_i)-f\bigglr(){\sum\limits_i
      PU_i}}_{\Omega\setminus \cup_iB(\xi_i,\eta),{\frac{2n}{ n+2}}} \\
  =O\bigglr(){
\sum_i\abs{U_i}^{\pc-1} _{\Omega\setminus  B(\xi_i,\eta),{\frac{2n}{ n-2}}}}
  =O\left(
\sum_i  \left({\frac{\delta_i}{\eta}}\right)^n\right)\end{multline}
and if $j\not=i$
\begin{align} \label{re1.10}
  &\abs{f(PU_j)}_{ B(\xi_i,\eta),{\frac{2n}{ n+2}}} =\abs{U_j}^{\pc-1}_{ B(\xi_i,\eta),{\frac{2n}{ n-2}}} =O\left(
\abs{U_j}^{\pc-1} _{\Omega\setminus  B(\xi_j,\eta),{\frac{2n}{ n-2}}}\right) =O\left(
   \left({\frac{\delta_j}{\eta}}\right)^{\frac{n+2}{2}}\right).\end{align}
   Moreover
 \begin{align*}
  &\biggabs{f(PU_i)-f\bigglr(){\sum\limits_i PU_i}}_{ B(\xi_i,\eta),{\frac{2n}{ n+2}}}\\
  &=O\xlr(){\abs{U_i^{\pc-2} (PU_i-U_i)}_{ B(\xi_i,\eta),{\frac{2n}{ n+2}}}
   }+O\bigglr(){ \sum_{j\not=i}\abs{U_i^{\pc-2}   U_j}_{ B(\xi_i,\eta),{\frac{2n}{ n+2}}}
   }\\
  &+O\left(\abs{| PU_i-U_i |^{\pc-1}}_{ B(\xi_i,\eta),{\frac{2n}{ n+2}}}
   \right)+O\bigglr(){ \sum_{j\not=i}\abs{U_j ^{\pc-1}}_{ B(\xi_i,\eta),{\frac{2n}{ n+2}}}
   }
   \end{align*}
and the first term is estimated in  \eqref{re1.7}, the third term is estimated in  \eqref{re1.6}, the fourth term is estimated in \eqref{re1.10}. The second term is estimated using \eqref{re1.7.1} and \eqref{re1.7.2} (with $q\sim1$ when $n\le6$ or $q\sim(n+2)/8$ when $n\ge7$)  as follows
    \begin{multline}\label{re1.11}
  \abs{U_i^{\pc-2}   U_j}_{ B(\xi_i,\eta),{\frac{2n}{ n+2}}}\\
  \begin{aligned}
    &=\delta_i^{\frac{n-2}{2}}  
    O\left(\abs{U_i}^{\pc-2}_{ B(\xi_i,\eta),{\frac{8nq}{
            (n-2)(n+2)}}} \right)
    O\left(\xabs{{\frac{1}{ |x- \xi_j|^{n-2}}}}_{ B(\xi_i,\eta),{\frac{2nq}{ (q-1)(n+2)}}} \right)   \\
    &=
    \begin{cases}
      \displaystyle O\left( \left({\frac{\delta_i}{ \epsilon}}\right) ^{
          {\frac{n+2}{2}}-\sigma}\right)&\qquad\text{if } n\ge7\\[3ex]
        \displaystyle O\left(\left({\frac{\delta_i}{ \epsilon}}\right) ^{ {n-2}-\sigma} \right)
       &\qquad\text{if } n\le6,
    \end{cases}
  \end{aligned}
\end{multline}
for some $\sigma>0.$

 Arguing exactly as in Proposition 2 of \cite{MR1133750}, we can estimate  the last term $I_4$ by
\begin{equation} \label{re1.12}\bigabs{
 a(x)\left[f \left(V_{ \si,\di,\ti}\right)-f_\eps\left(V_{ \si,\di,\ti}\right)\right]}_{\frac{2n}{ n+2}}= O\left(\epsilon|\ln\epsilon|\right).\end{equation}

\end{proof}
 \section{An estimate of the energy}

It is standard  to prove that
\begin{equation*}
  \widetilde J_\eps({\si,\di,\ti})=J_\eps\left(V_{\si,\di,\ti}\right)+\text{h.o.t.}
\end{equation*}
(see for example \cite{MR2232205} or  \cite{MR1384837}),
so the problem reduces to estimating the leading term $J_\eps\left(V_{\si,\di,\ti}\right).$
We will estimate the leading term in case \eqref{ans2}, because the expansion of the leading term in case \eqref{ans1} is easier and can be deduced from that.
We also assume $\ell=2,$ because with some minor modifications we treat the general case. Therefore, the estimate will be a direct consequence of  Lemma \eqref{lex1} and Lemma \eqref{leq1}.

For future reference we define the constants
\begin{align}\label{g1}
  &\gamma_1=\alpha_n^{\pc}\int\limits_{\rr^n}{\frac{1}{ (1+|y|^2)^n}}dy,\\
\label{g2}
  &\gamma_2=\alpha_n^{\pc}\int\limits_{\rr^n}{\frac{1}{ (1+|y|^2) ^{\frac{n+2}{2}}}}dy,\\
\label{g3}
  &\gamma_3=\alpha_n^{\pc }\int\limits_{\rr^n}{\frac{1}{ (1+|y|^2) ^{n }}}\log {\frac{1}{ (1+|y|^2 )^{\frac{n-2}{2}}}}dy.
\end{align}

We start with the following key estimates.
\begin{lemma}\label{lew1}
The following estimate holds true:
\begin{equation}\label{eq:100}
  \int\limits_{B(\xi_1,\eta)} a(x)U_1 ^{\pc }dx=\gamma_1a(s_1)+\<\nabla a(s_1),\nu(s_1)\>\gamma_1t_1\epsilon+O\left(\epsilon^2\right).
\end{equation}
Here $\eta$ is choosen as in \eqref{eta}.
\end{lemma}
\begin{proof}
 We split the left-hand side as
\begin{align}\label{lew13}
  &  \int\limits_{B(\xi_1,\eta)}  a(x) U_1 ^{\pc }dx= \int\limits_{B(\xi_1,\eta)} a(s_1)U_1 ^{\pc }dx+\int\limits_{B(\xi_1,\eta)}\left(a(x)-a(\xi_1)\right)U_1 ^{\pc }dx.\end{align}
We deduce
\begin{align}\label{lew14}
\int\limits_{B(\xi_1,\eta)} a(\xi_1)U_1 ^{\pc }dx=\gamma_1a(\xi_0)+O\left({\frac{\delta_1^n}{\eta^n}}\right).\end{align}

By     the mean value theorem we get
\begin{equation}\label{lew15} 
  \begin{aligned}
    a(\delta _1 y+\xi_1 )-a(\xi_1)&=
    a(\delta _1 y+\eta_1\nu(s_1)+s_1)-a(\xi_0)\\
    &=\<\nabla a(s_1 ),\nu(s_1)\>\eta_1+\delta_1 \<\nabla
    a(s_1),y\>+R(y),
  \end{aligned}
\end{equation}
where $R$ satisfies the uniform estimate
\begin{align}\label{lew16}
|R(y )|\le c\left(\delta_1^2|y|^2+\delta_1\eta_1|y|+\eta_1^2\right)\ \hbox{for any}\ y\in B(0,\eta/\delta_1).\end{align}
Therefore we conclude
\begin{multline}\label{lew17}
   \int\limits_{B(\xi_1,\eta)}\left(a(x)-a(\xi_1)\right)U_1 ^{\pc }dx\\
  \begin{aligned}
    &= \alpha_n^{\pc }\int\limits_{B(0,\eta/\delta_1)}\left[a(\delta y+\eta_1\nu(s_1)+s_1)-a(s_1)\right]{\frac{1}{(1+|y|^2)^{n}}}dy\\
    &=\alpha_n^{\pc }\int\limits_{B(0,\eta/\delta_1)}\left[\<\nabla a( s_1),\nu(s_1)\>\eta_1+\delta_1  \<\nabla a(s_1),y\>+R(y)\right]{\frac{1}{(1+|y|^2)^{n}}}dy\\
    & =\<\nabla a(s_1),\nu(s_1)\>\gamma_1\eta_1+O(\eta_1^2).
  \end{aligned}
\end{multline}

\end{proof}

\begin{lemma}\label{lez1}
The following estimates hold true:
   \begin{equation}
     \int\limits_{  B(\xi_1,\eta) }a(x)    U_1  ^ {\pc-1}\left(
       PU_1-U_1\right)dx
     =-\gamma_2a(s_1)\epsilon \left({\frac{d_1}{2t_1}}\right)^{n-2}+O\left(\epsilon^{1+\sigma}\right) \label{eq:10}
   \end{equation}
and
\begin{multline}
  \int\limits_{  B(\xi_1,\eta) }a(x)   U_1  ^ {\pc-1}PU_2dx \\
  = \left\{
    \begin{aligned}
      &O\left(\epsilon^{1+\sigma}\right)&&\qquad\text{if} s_1\not= s_2,\\
      &\gamma_2a(\xi_0)\epsilon\left( d_1d_2 \right)^{\frac{n-2}{2}}\times\\
      &\hspace{.5cm}\times
      \left({\frac{1}{|t_1-t_2|^{n-2}}}-{\frac{1}{|t_1+t_2|^{n-2}}}\right)+O\left(\epsilon^{1+\sigma}\right)
      &&\qquad\text{if } s_1 = s_2 =\xi_0,\\
    \end{aligned}\right.
  \label{eq:11}
\end{multline}
for some $\sigma>0.$
Here $\eta$ is choosen as in \eqref{eta}.
\end{lemma}
\begin{proof}
  First we prove~\eqref{eq:10}.  By Lemma \ref{lem1} and Lemma
  \ref{lem2} we get
\begin{multline}\label{lez11}
   \int\limits_{B(\xi_1,\eta )} a(x)U_1^{\pc-1} \left(
     PU_1-U_1\right)dx\\
   \begin{aligned}
     &=  \int\limits_{B(\xi_1,\eta )} a(x)U_1^{\pc-1} \left( -\alpha_n\delta_1^{\frac{n-2}{ 2}}H(x,\xi_1)+R_{\delta_1,\xi_1}\right)dx \\
     & =
     -\alpha_n^{\pc }  \delta_1  ^{n-2} \int\limits_{B(0,\eta
       /\delta_1)}a(\delta_1y+\xi_1) H(
     \delta_1y+\xi_1,\xi_1){\frac{1}{ (1+|y|^2)^{\frac{n+2}{2}}}}dy\\
     &\hspace{3cm}+ O\left( \left({\frac{\delta_1}{\eta_1}}\right)  ^{n} \right) \\
     &=-\alpha_n^{\pc }\delta_1^{n-2}\int\limits_{B(0,\eta /\delta_1)} a(\delta_1y+\xi_1 ){\frac{1}{ |\delta_1y+\xi_1-\bar\xi_1|^{n-2}}} {\frac{1}{ (1+|y|^2)^{\frac{n+2}{2}}}}dy \\
     &\hspace{3cm}+O\left( \left({\frac{\delta_1}{\eta_1}}\right) ^{n-2}\eta_1
     \right)+
     O\left( \left({\frac{\delta_1}{\eta_1}}\right)  ^{n} \right) \\
     & = -\alpha_n^{\pc }\left({\frac{\delta_1}{2\eta_1}}\right)^{n-2}
     {a(s_1) }\int\limits_{\rr^n }
     {\frac{1}{(1+|y|^2)^{\frac{n+2}{2}}}}dy\\
     &\hspace{3cm}+O\left(
       \left({\frac{\delta_1}{\eta_1}}\right) ^{n-1} \right)+ O\left(
       \left({\frac{\delta_1}{\eta_1}}\right) ^{n-2}\eta_1 \right),
   \end{aligned}
\end{multline}
because
\begin{equation*}
  |\delta_1y+\xi_1-\bar\xi_1|=|\delta_1y+2\eta_1 \nu(s_1)|\ge  2 \eta_1 -|\delta_1y|\ge \eta_1\ \hbox{for any} \ y\in B(0,\eta /\delta_1).
\end{equation*}
and by mean value theorem
\begin{equation*}
  a(\delta_1y+\xi_1)=a(s_1)+O\left(\eta_1\right)\ \hbox{and}\ {\frac{1}{ |\delta_1y+\xi_1-\bar\xi_1|^{n-2}}}={\frac{1}{ (2\eta_1)^{n-2}}}+O\left({\frac{\delta_1|y|}{ \eta_1^{n-1}}}\right).
\end{equation*}

Next, we prove~\eqref{eq:11}.
By Lemma \ref{lem2}
\begin{multline}\label{lez111}
   \int\limits_{B(\xi_1,\eta )}  a(x)  U_1^{\pc-1}  PU_2dx\\
   \begin{aligned}
     &=\int\limits_{B(\xi_1,\eta )}    a(x) U_1^{\pc-1} \left(U_2-\alpha_n\delta_2^{\frac{n-2}{2}}H(x,\xi_2)+R_{\delta_2,\xi_2}\right)dx \\
     & =\alpha_n^{\pc }
     (\delta_1\delta_2)^{\frac{n-2}{2}}\int\limits_{B(0,\eta
       /\delta_1)}a(\delta_1y+\xi_1){\frac{1}{(1+|y|^2)^{\frac{n+2}{2}}}}
     \times\\
     &\hspace{1cm} \times \left({\frac{1}{(\delta_2^2+|\delta_1y+\xi_1-\xi_2|^2)^{\frac{n-2}{2}}}}dy - H(\delta_1y+\xi_1,\xi_2) \right)dy \\
     &\hspace{1cm}+O\left(   (\delta_1\delta_2)^{\frac{n-2}{2}} {\frac{\delta_2^{\frac{n+2}{2}} }{\eta_2 ^{n } }}\right) \\
     &=\alpha_n^{\pc } (\delta_1\delta_2)^{\frac{n-2}{2}}\int\limits_{B(0,\eta /\delta_1)}a(\delta_1y+\xi_1){\frac{1}{(1+|y|^2)^{\frac{n+2}{2}}}}\times \\
     &\hspace{1cm} \times
     \left({\frac{1}{(\delta_2^2+|\delta_1y+\xi_1-\xi_2|^2)^{\frac{n-2}{2}}}}-{\frac{1}{ |\delta_1y+\xi_1-\bar\xi_2| ^{n-2 }}}\right)dy\\
     & \hspace{1cm}+
     O\left(  {\frac{(\delta_1\delta_2)^{\frac{n-2}{2}}  }{\eta ^{n-2} }} \eta_2  \right) +O\left(   (\delta_1\delta_2)^{\frac{n-2}{2}} {\frac{\delta_2^{\frac{n+2}{2}} }{\eta_2 ^{n } }}\right)   \\
     & =\alpha_n^{\pc } (\delta_1\delta_2)^{\frac{n-2}{2}}a(\xi_0)
     \left({\frac{1}{ | \eta_1-\eta_2| ^{n-2 }}}-{\frac{1}{ |
           \eta_1+\eta_2| ^{n-2 }}}\right) \int\limits_{\rr^n}
     {\frac{1}{(1+|y|^2)^{\frac{n+2}{2}}}}
     dy\\
     & \hspace{1cm}+O\left( {\frac{(\delta_1\delta_2)^{\frac{n-2}{2}} }{\eta
           ^{n-1} }} \right) +O\left( {\frac{
           (\delta_1\delta_2)^{\frac{n-2}{2}} }{\eta ^{n -1}
         }}\delta_1\right)
     +O\left(   (\delta_1\delta_2)^{\frac{n-2}{2}} {\frac{\delta_2^{\frac{n+2}{2}} }{\eta_2 ^{n } }}\right)  , 
   \end{aligned}
    \end{multline}
because for any $\ y\in B(0,\eta /\delta_1)$ we have
\begin{equation*}
  |\delta_1y+\xi_1-\bar\xi_2|=|\delta_1y+(\eta_1 +\eta_2)\nu(\xi_0)|\ge    \eta_1 +\eta_2-|\delta_1y|\ge \eta 
\end{equation*}
\begin{equation*}
  |\delta_1y+\xi_1- \xi_2|  \ge    |\xi_1- \xi_2|-|\delta_1y|\ge \eta  
\end{equation*}
and by mean value theorem
 $a(\delta_1y+\xi_1)=a(\xi_0)+O\left(\eta_1\right)$
 and
\begin{align*}
   &{\frac{1}{(\delta_2^2+|\delta_1y+\xi_1-\xi_2|^2)^{\frac{n-2}{2}}}}-{\frac{1}{ |\delta_1y+\xi_1-\bar\xi_2| ^{n-2 }}}\\
  &=
 {\frac{1}{ |\eta_1-\eta_2| ^{n-2 }}}-{\frac{1}{|\eta_1+\eta_2| ^{n-2 }}}+
 O\left({\frac{\delta_1|y|+\delta_2^2}{ \eta ^{n-1}}}\right) .
\end{align*}
 \end{proof}

\begin{lemma}\label{lex1}
The following estimate holds true:
\begin{multline*}
  J_0\left( V_{\si,\di,\ti}\right)
  =  {\frac{2-p}{2p}}\left[2\gamma_1a(\xi_0)+\gamma_1\<\nabla a(\xi_0),\nu(\xi_0)\>\epsilon(t_1+t_2)\right] 
  +{\frac{1}{2}}\gamma_2a(\xi_0)\times\\
  \times\left[\left({\frac{d_1}{2t_1}}\right)^{n-2}+\left({\frac{d_2}{2t_2}}\right)^{n-2}
+2\left(d_1d_2\right)^{\frac{n-2}{2}}\left({\frac{1}{|t_1-t_2|^{n-2}}}-{\frac{1}{|t_1+t_2|^{n-2}}}\right)\right]\epsilon\\
   +O\left(\epsilon^{1+\sigma}\right),
\end{multline*}
for some $\sigma>0.$
\end{lemma}
\begin{proof}
\begin{align}\label{lex10}
  &J_0\left( V_{\si,\di,\ti}\right)={\frac{1}{2}}\int\limits_\Omega a(x)|\nabla V_{\si,\di,\ti}|^2dx -{\frac{1}{p}}\int\limits_\Omega a(x)|  V_{\si,\di,\ti}|^pdx\end{align}

\medskip
We estimate the first term at the R.H.S. of \eqref{lex10}.
We write
\begin{multline}\label{lex11}
\int\limits_\Omega a(x)|\nabla V_{ \di,\ti}|^2dx\\
=\int\limits_\Omega a(x)|\nabla PU_1|^2dx+\int\limits_\Omega a(x)|\nabla PU_2|^2dx-2\int\limits_\Omega a(x)\nabla PU_1\nabla PU_2dx
\end{multline}

Let us estimate the first term in \eqref{lex11}. The estimate of the second term   is similar.
Let us choose $\eta$ as in \eqref{eta}.
We get
\begin{equation}\label{lex12}
  \begin{aligned}
     \int\limits_\Omega a(x)|\nabla PU_1|^2dx
     &= -\int\limits_\Omega \opdiv\left(a(x) \nabla PU_1\right)PU_1dx\\
    &=
    - \int\limits_\Omega \left(a(x)\Delta PU_1\right)PU_1dx- \int\limits_\Omega\<\nabla a,  \nabla PU_1 \>PU_1dx\\
    &= \int\limits_\Omega a(x) U_1^{\pc-1}PU_1dx - \int\limits_\Omega\<\nabla a,  \nabla PU_1 \>PU_1dx\\
    &= \int\limits_{ B(\xi_1,\eta )} a(x) U_1^{\pc-1}PU_1dx+
    \int\limits_{\Omega\setminus B(\xi_1,\eta )} a(x)
    U_1^{\pc-1}PU_1dx\\
    &\hspace{4.5cm}- \int\limits_\Omega\<\nabla a, \nabla PU_1
    \>PU_1dx
  \end{aligned}
\end{equation}
 By \eqref{cru2} we  deduce for some $\beta,\sigma>0$
 \begin{equation}\label{lex13.1}
     \int_\Omega\<\nabla a, \nabla PU_1 \>PU_1dx
     \le C\int_\Omega\abs{\nabla PU_1}PU_1dx\\
     =O\xlr(){\delta_1^{\frac{n-2}{n-1}+\beta}}
     =O\xlr(){\epsilon^{1+\sigma}}.
\end{equation}

 By Lemma \ref{lem2} we also deduce

\begin{align}\label{lex13}
 \int\limits_{\Omega\setminus B(\xi_1,\eta )} a(x)  U_1^{\pc-1}PU_1dx=O\left(\left({\frac{\delta_1}{\eps  }}\right)^n\right)
\end{align}
and
\begin{align}\label{lex14}
  &\int\limits_{B(\xi_1,\eta )}  a(x)U_1^{\pc-1}PU_1dx=    \int\limits_{B(\xi_1,\eta )}  a(x)U_1^{\pc }dx +\int\limits_{B(\xi_1,\eta )}  a(x)U_1^{\pc-1}\left(PU_1-U_1\right)dx.
\end{align}
  The first term is estimated in Lemma \ref{lew1} and the second term
  is estimated in \eqref{eq:10} of Lemma \ref{lez1}.

It remains only to estimate the last term in \eqref{lex11}.

\begin{equation}\label{lex18.1}
  \begin{aligned}
    \int\limits_\Omega a(x)\nabla PU_1\nabla PU_2dx
    &=-\int\limits_\Omega \opdiv\left(a\nabla PU_1\right) PU_2dx\\
    & =-\int\limits_\Omega \left(a\Delta PU_1\right) PU_2dx
    -\int\limits_\Omega  \<\nabla a,\nabla PU_1\> PU_2dx\\
    & =\int\limits_\Omega a(x) U_1^{\pc-1} PU_2dx-\int\limits_\Omega
    \<\nabla a,\nabla PU_1\> PU_2dx .
  \end{aligned}
\end{equation}

We have
\begin{align}\label{lex18}
  & \int\limits_\Omega  a(x) U_1^{\pc-1}  PU_2dx =\int\limits_{B(\xi_1,\eta )}\dots+\int\limits_{\Omega\setminus B(\xi_1,\eta )}\dots
\end{align}
and
\begin{multline}\label{lex19}
    \int\limits_{\Omega\setminus B(\xi_1,\eta )} a(x) U_1^{\pc-1}
    PU_2dx\\
  \begin{aligned}
    &=O\left({\delta_1^{\frac{n+2}{2}}\delta_2^{\frac{n-2}{2}} } \int\limits_{\Omega\setminus B(\xi_1,\eta )}{\frac{1}{|x-\xi_1|^{n+2}}}{\frac{1}{|x-\xi_2|^{n-2}}}dx\right)\\
    &
    =O\left({\frac{\delta_1^{\frac{n+2}{2}}\delta_2^{\frac{n-2}{2}}}{\eta^n
        }} \int\limits_{\rr^n\setminus B(0,1
        )}{\frac{1}{|y|^{n+2}}}{\frac{1}{|y+{\frac{\xi_1-\xi_2}{\eta}}|^{n-2}}}dy\right)\\
    &=O\left({\frac{\delta_1^{\frac{n+2}{2}}\delta_2^{\frac{n-2}{2}}}{\eta^n }}\right)
  \end{aligned}
\end{multline}

 The first term in \eqref{lex18} is estimated in \eqref{eq:11} of Lemma \ref{lez1}.

Finally,  as in the proof of \eqref{lex13.1}, from
\eqref{cru2} we obtain
\begin{equation}\label{lex20}
     \int_\Omega\<\nabla a, \nabla PU_2 \>PU_1dx
     =O\xlr(){\epsilon^{1+\sigma}},
\end{equation}
since $0<C_1\le\delta_2/\delta_1\le C_2$ on compact subsets of $\Lambda$.

 \medskip
We estimate the second term at the R.H.S. of \eqref{lex10}.
We write
\begin{equation}\label{ley11}
  \begin{aligned}
    \int\limits_\Omega  a(x)|  V_{ \di,\ti}|^{\pc }dx
    &=\int\limits_\Omega  a(x)|  PU_1-PU_2|^{\pc }dx\\
    &
    =\int\limits_\Omega  a(x)\left(|  PU_1-PU_2|^{\pc }-|   U_1 |^{\pc }-|    U_2|^{\pc }\right)dx\\
    & \hspace{3cm}+\int\limits_\Omega a(x)\left( | U_1 |^{\pc } +| U_2|^{\pc
      }\right)dx.
  \end{aligned}
\end{equation}

The last two terms in \eqref{ley11} are estimated in Lemma \ref{lew1}. Let us choose $\eta$ as in \eqref{eta}.

We split the first integral as
\begin{multline}\label{ley12}
    \int\limits_\Omega  a(x)\left(|  PU_1-PU_2|^{\pc }-|   U_1 |^{\pc }-|    U_2|^{\pc }\right)dx\\
=\int\limits_{B(\xi_1,\eta)}\dots
+\int\limits_{B(\xi_2,\eta)}\dots
+\int\limits_{\Omega\setminus\left(B(\xi_1,\eta)\cup B(\xi_2,\eta)\right)}\dots
\end{multline}

From Lemma \ref{lem2} we deduce

\begin{multline}\label{ley13}
     \int\limits_{\Omega\setminus\left(B(\xi_1,\eta)\cup B(\xi_2,\eta)\right)}  a(x)\left(|  PU_1-PU_2|^{\pc }-|   U_1 |^{\pc }-|    U_2|^{\pc }\right)dx\\
  =O\left( \int\limits_{\Omega\setminus\left(B(\xi_1,\eta)\cup B(\xi_2,\eta)\right)}  \left(    U_1  ^{\pc }+ U_2 ^{\pc }\right)dx\right)=O\left({\frac{\delta_1^n}{\eta^n}}+{\frac{\delta_2^n}{\eta^n}}\right).
\end{multline}

We now estimate the integral over $B(\xi_1,\eta)$.
\begin{multline}\label{ley14}
    \int\limits_{  B(\xi_1,\eta) }  a(x)\left(|  PU_1-PU_2|^{\pc }-|   U_1 |^{\pc }-|    U_2|^{\pc }\right)dx\\
    \begin{aligned}
      &=
      \pc  \int\limits_{  B(\xi_1,\eta) }  a(x)    U_1  ^ {\pc-1}\left(  PU_1-U_1 - PU_2\right) dx\\
      &
      +{\frac{(\pc-1)\pc}{2}} \int\limits_{  B(\xi_1,\eta) }  a(x)   | U_1+\theta\left( PU_1-U_1-PU_2\right)|^{\pc-2} \left(  PU_1-U_1-PU_2\right)^2dx\\
      &
      -\int\limits_{  B(\xi_1,\eta) }  a(x) |    U_2|^{\pc } dx\\
      &= \pc \int\limits_{ B(\xi_1,\eta) } a(x) U_1 ^ {\pc-1}\left(
        PU_1-U_1-PU_2\right)dx+I,
    \end{aligned}
\end{multline}
where $I $ is defined and estimated as
\begin{equation}\label{ley15}
  \begin{aligned}
    & I := {\frac{(\pc-1)\pc }{2}} \int\limits_{ B(\xi_1,\eta) } a(x)
    | U_1+\theta\left( PU_1-U_1-PU_2\right)|^{\pc-2} \left(
      PU_1-U_1-PU_2\right)^2dx
    \\
    &-\int\limits_{  B(\xi_1,\eta) }  a(x) |    U_2|^{\pc } dx\\
    &= O\left(\int\limits_{ B(\xi_1,\eta) } U_1 ^{\pc-2} \left(
        PU_1-U_1 \right)^2dx\right)+O\left(\int\limits_{ B(\xi_1,\eta)
      } U_1 ^{\pc-2} U_2 ^2dx\right)
    \\
    &+ O\left(\int\limits_{  B(\xi_1,\eta) }|PU_1-U_1 |^{\pc }\right)+O\left(\int\limits_{  B(\xi_1,\eta) }  |    U_2|^{\pc } dx\right)\\
    &= O\left(\abs{U_1 ^{\pc-2} \left( PU_1-U_1
        \right)}_{B(\xi_1,\eta),{\frac{2n}{
            n+2}}}\abs{PU_1-U_1}_{B(\xi_1,\eta),{\frac{2n}{
            n-2}}}\right)
    \\
    &+O\left(\abs{U_1 ^{\pc-2} U_2}_{B(\xi_1,\eta),{\frac{2n}{
            n+2}}}\abs{U_2}_{B(\xi_1,\eta),{\frac{2n}{ n-2}}}\right)
    \\
    &+O\left(\abs{PU_1-U_1}^{\pc }_{B(\xi_1,\eta),{\frac{2n}{ n-2}}}\right) +O\left(\abs{U_2}^{\pc }_{B(\xi_1,\eta),{\frac{2n}{ n-2}}}\right)\\
    & =O\left(\eps^{1+\sigma}\right),
  \end{aligned}
\end{equation}
for some $\sigma>0$, because of  estimates \eqref{re1.6}, \eqref{re1.7}, \eqref{re1.10} and \eqref{re1.11}.

The first term in \eqref{ley14} is estimated in \eqref{eq:10} and \eqref{eq:11} of Lemma \ref{lez1}.

Finally, we estimate the integral over $B(\xi_2,\eta)$
\begin{equation}\label{ley17}
  \begin{aligned}
    &   \int\limits_{  B(\xi_2,\eta) }  a(x)\left(|  PU_1-PU_2|^{\pc }-|   U_1 |^{\pc }-|    U_2|^{\pc }\right)dx\\
    &=
    -\pc  \int\limits_{  B(\xi_2,\eta) }  a(x)    U_2  ^ {\pc-1}\left( -PU_2+U_2+PU_1\right)dx\\
    &
    +{\frac{(\pc-1)\pc }{2}} \int\limits_{  B(\xi_2,\eta) }  a(x)   | U_1+\theta\left( -PU_2+U_2+PU_1\right)|^{\pc-2}\left( -PU_2+U_2+PU_1\right)^2dx\\
    &
    -\int\limits_{  B(\xi_2,\eta) }  a(x) |    U_1|^{\pc } dx\\
    &= \pc \int\limits_{ B(\xi_2,\eta) } a(x) U_2 ^ {\pc-1}\left(
      PU_2-U_2-PU_1\right)dx+J ,
  \end{aligned}
\end{equation}
where $J $ is estimated exactly as  in \eqref{ley15},
while the first term  in \eqref{ley17} is estimated in \eqref{eq:10}
and \eqref{eq:11} of Lemma \ref{lez1}.

\medskip
We collect all the previous estimates  and we get the claim.
\end{proof}

\begin{lemma}\label{leq1}
The following estimate holds true:
\begin{multline}
\frac{1}{ p-\eps}\int\limits_\Omega a(x)|V_{\si, \di,\ti}|^{\pc
  -\eps}dx
  =\frac{1}{ p }\int\limits_\Omega a(x)|V_{\si, \di,\ti}|^{\pc  }dx\\
\begin{aligned}
  &\hspace{1cm}+   \eps\left[{\frac{ 1}{ \pc ^2}}\int\limits_\Omega a(x)|V_{\si, \di,\ti}|^{\pc }dx-{\frac{ 1}{  \pc  }}\int\limits_\Omega a(x)|V_{\si, \di,\ti}|^{\pc-1}\log |V_{\si, \di,\ti}|dx\right]+o(\eps)\\
  &= \left[a(s_1)+a(s_2)\right]\left({\frac{\gamma_1}{\pc
        ^2}}-{\frac{\gamma_1\alpha_n}{ \pc }}-{\frac{\gamma_3}{ \pc }}
  \right)\eps\\
  &\hspace{4cm}+{\frac{n-2}{2\pc }} \gamma_1 \left[a(s_1)\log\delta_1+
    a(s_2)\log\delta_2\right]\eps+o(\eps).
\end{aligned}
\end{multline}
\end{lemma}
\begin{proof}
We
argue exactly as in the proof of Lemma 3.2 of \cite{MR1979006}.
 \end{proof}

\bibliographystyle{amsplain-abbrv} \bibliography{aclapi-1}

\end{document}